\documentclass[]{article}
\usepackage[cp1251]{inputenc}
\usepackage{amssymb}
\usepackage{amsmath}
\usepackage{authblk}
\usepackage{amscd}
\usepackage{mathtools}
\usepackage{amsthm}
\usepackage{color}
\usepackage[
top    = 3.00cm,
bottom = 3.00cm,
left   = 3.00cm,
right  = 2.80cm]{geometry}
\usepackage[style=alphabetic,natbib=true,bibencoding=ascii,backend=bibtex]{biblatex}
\usepackage{hyperref}
\usepackage{etoolbox}
\usepackage{tikz-cd}
\usepackage{xfrac}
\usepackage{nicefrac}
\usepackage{marvosym}
\usepackage{faktor}
\usepackage{makeidx}
\usepackage{tikz-cd}
\apptocmd{\thebibliography}{\raggedright}{}{}

\usepackage{titletoc}

\dottedcontents{section}[1.5em]{\bfseries}{1.3em}{.6em}

\newtheorem{theorem}{Theorem}[section]
\newtheorem{lemma}[theorem]{Lemma}
\newtheorem{proposition}[theorem]{Proposition}
\newtheorem{corollary}[theorem]{Corollary}
\newtheorem{remark}{Remark}
\newtheorem{definition}[theorem]{Definition}
\newtheorem{example}{Example}[section]

\title{On the existence of analytic families of $G$-stable lattices and their reductions}
\author{Emiliano Torti}
\affil{torti.emiliano@gmail.com}
\affil{Laboratoire GAATI, Universit\'e de la Polyn\'esie fran\c{c}aise}

\begin{document}
	
	\maketitle

\begin{abstract}\noindent
In this article, we prove the existence of rigid analytic families of $G$-stable lattices with locally constant reductions inside families of representations of a topologically compact group $G$, extending with a different strategy a result of Hellman obtained in the semi-simple residual case. Implementing this generalization in the context of Galois representations, we prove a local constancy result for reductions modulo prime powers of trianguline representations of generic dimension $d$ which justify the existence in the literature of similar results in dimension 2 for crystalline and semistable representations. Moreover, we present two explicit applications. First, in dimension two, we extend to a prime power setting and to the whole rigid projective line a recent result of Bergdall, Levin and Liu concerning reductions of semi-stable representations of $\text{Gal}(\overline{\mathbb{Q}}_p / \mathbb{Q}_p)$ with fixed Hodge-Tate weights and large $\mathcal{L}$-invariant. 
Second, in dimension $d$, let $V_n$ be a sequence of crystalline representations converging in a certain geometric sense to a crystalline representation $V$. We show that for any refined version $(V, \sigma)$ of $V$ (or equivalently for any chosen triangulation of its attached $(\varphi, \Gamma)$-module $D_{\text{rig}} (V)$ over the Robba ring), there exists a sequence of refinement $\sigma_n$ of each of the $V_n$ such that the limit as refined representations $(V_n , \sigma_n )$ converges to the $(V, \sigma)$. This result does not hold under the weaker assumption that $V_n$ converges only uniformly $p$-adically to $V$ (in the sense of Chenevier, Khare and Larsen). 
\end{abstract}

\section{Introduction}

Let $p$ be a prime number. The study and classification of reductions of representations of the absolute Galois group of $\mathbb{Q}_p$, denoted $G_{\mathbb{Q}_p}$, plays a central role in modern number theory from modularity lifting theorems to the study of slopes of classical or overconvergent modular forms and it is still an open problem. Moreover, such topic is of independent local interest and presents wild differences when is approached by restricting oneself to study reductions of certain subclasses of representations of $G_{\mathbb{Q}_p}$ in characteristic zero such as crystalline or semi-stable.\\
Numerous different strategies have recently been implemented to attack this problem and they share the common point which, in simple terms, consists of identifying certain $p$-adic parameters which explicitly classify to some extent either the representations of interest themselves or the corresponding object in an equivalent category. Letting such parameters vary $p$-adically leads to explicit conditions which allow us to either obtain a direct description of the reductions or, when this is not possible, to understand when small $p$-adic deformations of such parameters give the same reductions. We will refer generically to the latter phenomena as the local constancy property.\\
The several approaches present in the literature include the use of Fontaine and Kedlaya's theory of $(\varphi, \Gamma)$-modules in $p$-adic Hodge theory and its integral version via Wach modules (see for example \cite{BLZ04}, \cite{Ber12} and \cite{Tor22}) or more generally via Breuil-Kisin modules (see for example \cite{BM02}, \cite{BL20}, \cite{BLL23}, \cite{Li21} and \cite{Par17}); the $p$-adic and mod $p$ Langlands correspondence (see for example \cite{BG09}, \cite{Bha20}, \cite{Ars21b}, \cite{BGR18} and \cite{BG15}) and via deformation theory and universal analytic families (see for example \cite{Roz20}).
Already by looking at the seminal work of Breuil and Mezard concerning a classification of residual reduction of semi-stable representations of dimension 2 in low weights, it is evident the difficulty of performing explicit computations. This reflects the high dependency of the reductions on the parameters coming for example from integral $p$-adic Hodge theory, which is one of the reasons why most of the results available in the literature (as the ones mentioned above) concern the cases of low dimension, i.e. 2 and 3. The aim of this article is to prove that it is always possible to produce, in any dimension $d$, local constancy results for reductions of a large class of representations which includes the crystalline and semi-stable ones.\\ 
In a similar spirit of the deformation theory approach, the strategy that we decided to adopt consists of considering the large class of trianguline representations introduced by Colmez in dimension 2 (see \cite{Col08}) and later extended to the high dimensional case by Chenevier (see \cite{Che13}). These representations arise naturally in rigid analytic families and they can be used to interpolate for example crystalline or semi-stable representation, providing a new set of analytic parameters. A first (and to our knowledge only) instance of the use of such approach in this context is due to Berger in his work concerning local constancy result for residual semi-simple reductions of crystalline representations of dimension 2 when the Hodge-Tate weights vary $p$-adically (see \cite{Ber12}).\\
Understanding if and how it is possible to extend such approach to the more general context of congruences modulo prime powers among trianguline representations of any dimension $d$ has been the starting point of this work. This article is the natural continuation of the author's previous work \cite{Tor22} in which the strategy is successfully tested on the simpler task of extending the results of Berger mentioned above to a prime power setting but only for the 2-dimensional crystalline case.\\
The problem of proving the existence of congruences between representations interpolated by some analytic families can be expressed in a pure algebraic setting of finite-dimensional representations of certain topological groups. For example, thanks to the work of Hellmann (see \cite{He16}), we know that it is possible to prove a local constancy result in the semi-simple residual case for adic families of representations of a compact topological group.\\
The aim of this article is double-folded. First, introduce a notion of local constancy for reductions of families of representations which agrees to the one commonly used in the semi-simple residual case (see for example \cite{Ber12} and \cite{He16}), it is suitable to the more delicate prime power setting and finally which allows us to extend the result of Hellmann. Second, apply the results obtained in the algebraic setting to the context of trianguline Galois representations and deduce some new local constancy results as well as some new symmetries of the trianguline variety. We proceed now in describing the main results of this article. \\
Let $\mathfrak{X}$ be a reduced quasi-compact, quasi-separated $\mathbb{Q}_p$-rigid analytic space. The reduced hypothesis is assumed here for simplicity of expression and it will be removed in the rest of the article since being locally constant is a local property. Let $G$ be a compact topological group. Let $\mathbb{V}$ be a family of representations of $G$ over $\mathfrak{X}$, i.e. a locally free coherent $\mathcal{O}_\mathfrak{X}$-module of uniform rank $d\in\mathbb{Z}_{\geq 1}$ endowed with a continuous $\mathcal{O}_\mathfrak{X}$-linear action of $G$. It is well known (see for example Lemma 3.18 in \cite{Che20}) that $\mathbb{V}$ admits a $G$-stable integral subfamily $\mathbb{T}$ (seen as a locally free module over the subsheaf of power bounded elements $\mathcal{O}^\circ_{\mathfrak{X}}$) such that we have an isomorphism $\mathbb{T}[\frac{1}{p}]\cong \mathbb{V}$ of $\mathcal{O}_\mathfrak{X} [G]$-modules.\\
Let $V$ be a finite dimensional representations of $G$ over some finite extension of $\mathbb{Q}_p$. We denote by $\overline{V}^{\text{ss}}$ the semi-simplification of a residual reduction of $V$. By the Brauer-Nesbitt theorem, such residual representation does not depend on the choice of a stable lattice. Hence, it is not required to explicitly describe an integral substructure for a family $\mathbb{V}$ in order to give a satisfying notion of semi-simple residual local constancy. Indeed, intuitively we can say that $\mathbb{V}$ has locally constant semi-simplification modulo $p$ if the map which associates to each $x$ the representation $\overline{V}^{\text{ss}}_x$ is locally constant (or equivalently the map which associates to each $x$ the collection of characteristic polynomials as functions in $x$) where $V_x$ is the specialization of $\mathbb{V}$ at $x$. This is the definition which many authors refer to (see for example \cite{Ber12}, \cite{Bha20}, \cite{BL20}), included the result below. 
Helmann proved in the adic context the following (see \cite{He16}): 
\begin{theorem}
Let $\mathfrak{X}$ be an adic space locally of finite type over $\mathbb{Q}_p$ and let $\mathbb{V}$ be a family of representations of a compact group $G$ over $\mathfrak{X}$.
At any point $x\in \mathfrak{X}$, the semi-simplification of the residual reduction of the specialization of $\mathbb{V}$ at $x$ is locally constant, i.e. $(\overline{\mathbb{V}\otimes k(x)})^{\text{ss}}$ is locally constant. 
\end{theorem}
\noindent
Extending this result to a prime power setting presents many extra difficulties when compared to the residual case. First, since the semi-simplification process is not available anymore we need to keep at all time a precise control of the involved lattices. Moreover, in practice, we say that two $d$-dimensional representations $V$ and $V'$ of $G$ defined over the same finite extension of $\mathbb{Q}_p$ are congruent modulo $p^n$ if there exist lattices $T$ and $T'$ respectively in $V$ and $V'$ such that their reduction modulo $p^n$ are isomorphic as representations of $G$. Hence, this condition is not an equivalence relation (as in the semi-simple residual case) as it fails in being transitive.\\
Dealing with local constancy phenomena in a rigid analytic setting versus the adic setting has its advantages and disadvantages. In the adic setting, the main strong points consist on having directly an integral model and having a proper topology (rather than a G-topology)
and hence notion of locally constant functions makes sense without the need of clarification. Indeed, any continuous morphism from a topological space to a space with discrete topology is locally constant. This can be used as a direct argument when one wants to prove the local constancy of the semi-simple residual reduction of an adic family of representations as it is possible to directly apply that to the coefficients of characteristic polynomials. This constitutes the key argument in the proof of Hellmann's result. However, in a prime powers setting, control over the coefficients of the characteristic polynomial is not enough and we need to introduce and work with an integral structure that, contrary to the adic setting, can be controlled quite explicitly once its existence is proven in the rigid analytic context thanks to Raynaud and Berthelot's theory of admissible formal models. 
Moreover, there are some drawbacks of defining the local constancy property in terms of $p$-adic maps sending any point $x$ to its attached residual representation as introduced before, e.g. how one can define it in a prime power context when the points don't have the same definition field. In order to overcome all these difficulties, the rigid analytic setting come into help because it allows us to explicitly describe and control what it means to have congruent sections and moreover it allows in many examples to have a quite tight control on the local constancy radius, as we will see later. For these reasons, we considered the rigid analytic setting as our main setting but it is straightforward to translate all of the results in an adic setting.\\
Naively speaking, we say that the a family $\mathbb{T}$ is locally constant modulo $p^n$ (for some $n\in\mathbb{Z}_{\geq 1}$) if locally the push-forward of the $\mathcal{O}^\circ_\mathfrak{X} / p^n \mathcal{O}^\circ_\mathfrak{X} $-module $\mathbb{T}/p^n \mathbb{T}$ via the restriction on power bounded elements of an affinoid specialization morphism $e_{*} : \mathcal{O}^\circ_{\mathfrak{X}}  \rightarrow \mathcal{O}_{\mathbb{K}_*}$ has a constant action of $G$.
This definition will imply as a consequence the (maybe more intuitive) classical notion of local constancy which corresponds to ask that that for any point $x \in \mathfrak{X}$ there exists a open admissible neighborhood $\mathcal{U}^{(n)}$ defined over $\mathbb{K}_x$, such that for any $y\in \mathcal{U}^{(n)}$ (whose field of definition $\mathbb{K}_y$ is a finite extension of $\mathbb{K}_x$) we have that as $\mathcal{O}_{\mathbb{K}_y}$-modules, the lattices satisfy $\mathbb{T}_x \equiv \mathbb{T}_y \text{ mod }\pi_{y}^{\gamma(n)}$ where $\gamma$ is a positive arithmetic function such that $\gamma(1)=1$ and it depends only on the ramification of $\mathbb{K}_y$ over $\mathbb{K}_x$ and where $\pi_y$ denotes a uniformizer of $\mathbb{K}_y$. Note that the introduction of the function $\gamma$ is part of the novelty of this article because as we will see later on, it will allow us to deal with any congruence between lattices not defined on the same field of definition. Indeed, note that the hope of always getting $\gamma(n)= n \cdot e_{\mathbb{K}_y / \mathbb{K}_x}$ (or equivalently getting exactly a congruence modulo $p^n$ after the specialization) will be disappointed in general.  For the purposes of this introduction we limited ourself to give a generic notion of how the local constancy property modulo prime powers will look like and we will give a more precise geometric definition later in the article. 
Now, we state the first result of this article: 

\begin{theorem} \label{reductions}
Let $\mathfrak{X}$ be a quasi-separated, quasi-compact $\mathbb{Q}_p$-rigid analytic space. Let $\mathbb{V}$ be a family of representations of $G$ over $\mathfrak{X}$ and let $\mathbb{T}$ be an integral subfamily of representations of $G$ defined over an admissible integral model $\mathcal{X}$ in the sense of Raynaud (of generic fiber $\mathfrak{X}$) such that $\mathbb{T}[\frac{1}{p}]\cong \mathbb{V}$ as $\mathcal{O}_\mathfrak{X} [G]$-modules.
For every positive integer $n\in\mathbb{Z}_{\geq 1}$, the integral family $\mathbb{T}$ is locally constant modulo $p^n$.
\end{theorem}
\noindent
Note that in the residual case (i.e. $n=1$), we recover the result of Hellmann as well as its extension to the new residual non-semisimple case. Indeed, as we will see later in the article with an example in dimension two, we are able to distinguish residually locally constant families whose reductions is non-split.\\ 
As we will clarify later, the result above has been stated in an abstract algebraic setting and it can be applied as well for example in some global context, e.g. for rigid analytic families of modular forms (see for example \cite{NOR23}) or to $p$-adic system of eigenvalues attached to the automorphic representations of a unitary group. Indeed, a result concerning local constancy modulo prime powers in the weight has been proven by Chenevier (see \cite{Che20}) in the latter setting. However, this is for now beyond our goals in this article and we proceed by introducing the local setting in which we will work.\\ 
Let $G=G_{\mathbb{Q}_p}$ be the absolute Galois group of $\mathbb{Q}_p$ after having fixed an algebraic closure of $\mathbb{Q}_p$. Colmez introduced the notion of trianguline representations of $G_{\mathbb{Q}_p}$ (see \cite{Col08}), i.e. finite dimensional representation $V$ over a finite extension $\mathbb{E}$ of $\mathbb{Q}_p$ whose attached $(\varphi, \Gamma)$-module over the Robba ring $D_{\text{rig}} (V)$ can be obtained via consecutive extensions of $(\varphi, \Gamma)$-modules of rank 1. A choice of such extension is called a triangulation of $V$. It is well-known thanks to the groundbreaking work of Colmez (see \cite{Col08}) and Chenevier (see \cite{Che13}) that the class of trianguline representations of dimension $d\in\mathbb{Z}_{\geq 1}$, together with other data concerning the triangulation, is parametrized by a reduced, regular $\mathbb{Q}_p$-rigid analytic space of dimension $d(d+3)/2$ denoted $\mathcal{S}_d^{\square}$. We will describe in details such space later in the article. Moreover, there exists a naive open (i.e. arbitrary union of admissible opens) $\mathcal{S}^{\square, 0}_d$ which is the \'etale locus inside $\mathcal{S}^\square_d$ over which there exists a universal family $\mathbb{V}_{\text{tri}}$ of trianguline representations seen as a locally free coherent $\mathcal{O}_{\mathcal{S}^{\square,0}_d}$-module of uniform rank $d$ endowed with a linear action of the Galois group $G_{\mathbb{Q}_p}$ which satisfies the universal property that after specialization of $\mathbb{V}$ at any point in $\mathcal{S}^{\square, 0}_d$ one recovers exactly the trianguline representation corresponding to the point $x$. Chenevier (see Prop. 3.17 in \cite{Che13}) proved that the universal family $\mathbb{V}_{\text{tri}}$ admits an integral subfamily $\mathbb{T}_{\text{tri}}$ seen as a locally free (of uniform rank $d$) module over the subsheaf of power bounded sections. As a direct application of the local constancy result in this trianguline setting, we deduce the following:
\begin{proposition}
Let $d$ be a positive integer and let $\mathbb{T}:=\mathbb{T}_{\text{tri}}$, defined as a locally free coherent module over $\mathcal{S}^{\square, 0}_d$. For every $x\in S^{\square, 0}_d$ there exists a countable system of open affinoid neighborhoods $\{ \Omega^{(n)} \}_{n\geq 1}$ inside $S^{\square}_d$, a countable compatible set $\{\mathcal{A}^{(n)}\}$ where each $\mathcal{A}^{(n)}$ is a model of the $\mathbb{Q}_p$-affinoid algebra $\mathcal{O}(\Omega^{(n)})$ such that:

\begin{enumerate}
\item (\'etale) $\dots\subset\Omega^{(n)} \subset\Omega^{(n-1)}\subset\dots\subset\Omega=\Omega^{(0)}\subset S^{\square,0}_d$;
\item (specialization) for each $y\in\Omega^{(n)} $, we have an isomorphism $D_{\text{rig}} (\mathbb{T}_y \otimes_{\mathcal{A}^{(n)}} \mathcal{O}(\Omega^{(n)} )_y)\cong D_y $ of regular, rigidified trianguline $(\varphi, \Gamma)$-modules;
\item (local constancy modulo $p^n$) for each $n\geq 1$, we have that $\mathbb{T}$ is locally constant mod $p^n$ on $\Omega$ and constant modulo $p^n$ on $\Omega^{(n)}$.
\end{enumerate}
Here $D_y$ denotes the trianguline $(\varphi, \Gamma)$-module attached to the point $y$, whose definition field is denoted here $\mathcal{O}(\Omega^{(n)} )_y$ and does not depend on $n$.
\end{proposition}
\noindent
This result is a refinement of a strong result Chenevier (see \cite{Che13}) and it has a deep significance in terms of understanding the reductions of crystalline and semi-stable representations. As mentioned before, most of the results in the literature describing the reductions of certain classes of representations are of the local constancy type and are mainly limited to the cases of dimension 2 and 3 which are already very hard to deal with in terms of explicit computations. The novelty of this result resides in its transversal approach which allow us to prove that such congruences always will exists in any dimension and it will be always possible to find all these congruences even in a prime power setting because the existence of integral stable lattices in families is granted. Of course, the price to pay with this approach is that one loses control of the explicit parameters. However, all these result can be made explicit once one is able to understand an explicit description for the local constancy radius in the residual case. Indeed, it will turn out that the local constancy radius is a linear function in the exponent of the prime powers with known coefficient being the residual local constancy radius. As a motivating example for this, we will present an application of this result in the study of reductions of semi-stable representations of $G_{\mathbb{Q}_p}$ of dimension 2 with large $\mathcal{L}$-invariant and fixed Hodge-Tate weights. Indeed, very recently Bergdall, Levin and Liu (see \cite{BLL23}) proved that semi-simple residual local constancy holds for that specific class of semi-stable representations. In order to state their result properly we will need to introduce some notations (consistent as much as possible with \cite{BLL23}).\\
Let $v_p$ be the normalized $p$-adic valuation on $\overline{\mathbb{Q}}_p$. Let $\pi \in \overline{\mathbb{Q}}_p$ such that $\pi^2 =p$. Let $\mathbb{Q}_{p^2}$ be the unramified quadratic extension of $\mathbb{Q}_p$ of absolute Galois group $G_{\mathbb{Q}_{p^2}}$, denote by $\eta$ its quadratic character modulo $p$ and by $\omega_2$ a fundamental character of level 2 for $\mathbb{Q}_{p^2}$. Up by twisting by a crystalline (or equivalently semi-stable) character, every semi-stable non-crystalline 2-dimensional representation $V_{k, \mathcal{L}}$ depends on two parameters $k\in\mathbb{Z}_{\geq 2}$ and $\mathcal{L}\in \overline{\mathbb{Q}}_p$. We can extend the parameter $\mathcal{L}$ to live in the set $\mathbb{P}^1 (\overline{\mathbb{Q}}_p)$, by setting $V_{k, \infty}$ to be the unique crystalline representation whose trace of the crystalline Frobenius $a_p$ satisfies $a_p = \pi^{k-2} + \pi^k$.  For any $k$ and $\mathcal{L}$, denote by $\overline{V}^{\text{ss}}_{k, \mathcal{L}}$ the semi-simple residual reduction of $V_{k, \mathcal{L}}$.\\
Bergdall, Levin and Liu proved (see \cite{BLL23}) that, assuming $k\in\mathbb{Z}_{\geq 4}$, and $p\not=2$, if $v_p (\mathcal{L}) < 2-\frac{k}{2}-v_p ((k-2)!),$ then the reduction $\overline{V}^{\text{ss}}_{k, \mathcal{L}} $ does not depend on the parameter $\mathcal{L}$ and as a consequence it is isomorphic to the semi-simple residual reduction $\overline{V}_{k, \infty}$, whose explicit description is known to be the representation $ \text{Ind}_{\text{G}_{\mathbb{Q}_{p^2}}}^{\text{G}_{\mathbb{Q}_{p}}} (\omega_2^{k-1} \eta)$ (thanks to the early work of Berger, Li and Zhou in \cite{BLZ04} and its improved version by Bergdall and Levin in \cite{BL20}).\\
We are able to extend this result to a prime power setting and to the whole projective line where $\mathcal{L}$ varies. To be precise, we prove the following:
\begin{proposition}
Let $k\geq2$ be a fixed positive integer and let $p$ be a fixed prime. Let $\mathcal{L}_1$  be a $\mathbb{E}$-point in $\mathbb{P}^{1, \text{rig}}$, for some finite extension $\mathbb{E}$ of $\mathbb{Q}_p$. Inside $\mathbb{P}^{1, \text{rig}}$, there exists a compatible system of $\mathbb{E}$-subaffinoid closed balls $\{\mathbb{B}_{\mathcal{L}_1 , p^{- r_n}}\}_{n\geq 1}$ centered in $\mathcal{L}_1$ of radius $p^{- r_n} \in |\mathbb{E}^\times |$ and there exists a $G_{\mathbb{Q}_p}$-stable lattice $T_1$ inside $V_{k, \mathcal{L}_1}$ such that 
for any $\mathcal{L}_2 \in \mathbb{B}_{\mathcal{L}_1 , p^{- r_m}}$ then there exist a $G_{\mathbb{Q}_p}$-stable lattice $T_2 \subset V_{k, \mathcal{L}_2}$ which satisfy the following congruence
$$T_1 \equiv T_2 \mod \pi_{\mathcal{L}_2}^{\gamma(m)} \quad\quad \text{ as }\faktor{\mathcal{O}_{\mathbb{K}_{\mathcal{L}_2}}}{\pi_{\mathcal{L}_2}^{\gamma(m)} \mathcal{O}_{\mathbb{K}_{\mathcal{L}_2}}} [G_{\mathbb{Q}_p}]\text{-modules,}$$ 
with $\gamma(m) =e_{\mathbb{K}_{{\mathcal{L}_2}} / \mathbb{E}} (m-1)+1$ where $\mathbb{K}_{\mathbb{L}_2}$ is the finite extension of $\mathbb{E}$ over which the point $\mathcal{L}_2$ is defined.\\
Moreover, the local constancy radiuses satisfy the linear relation $r_n = r_1 +n$ for all $n\geq 1$.
\end{proposition}
\noindent
The different approach we adopt to prove the above proposition allows us not only to extend the result of Bergdall, Levin and Liu to reductions modulo prime powers but it gives some more info on the residual side too. Indeed, for example it includes the subtle case $p=2$,  which is particularly hard to deal with (even in the crystalline case) especially for explicit computations due to the extra conditions imposed by the fact that the topological group $\Gamma$ is not procyclic anymore (reason why such case has been avoided in \cite{Ber12} and in \cite{Tor22}). Technical problems arise as well when $p=2$ in the construction of certain Kisin modules as noted by Bergdall, Levin and Liu (see remark 1.4 in \cite{BLL23}), however all the trianguline setting and constructions work well in such case as already observed by Colmez in \cite{Col08}. To our knowledge, this is the only local constancy result known when $p=2$. 
In characteristic 2, the descriptions of $\overline{V}_{k, \mathcal{L}}$ and even of the crystalline representation $\overline{V}_{k, \infty}$ are unknown, however we deduce that such congruence still must exist. 
Finally, in the last part of the article, we will present some examples where we have to deal with those reductions which are non-semisimple, showing that for example in dimension two, it is possible to produce locally constant families for all the non-split reductions of semi-stable representations (whose existence is granted by Ribet's lemma). \\
The final result of this article that we want to present concerns instead crystalline representations of general dimension $d$. The problem that we want to study consists in understanding the positions of trianguline points on the Colmez and Chenevier's trianguline variety which corresponds to a certain fixed crystalline representation. To give a more precise idea, let $V$ be a crystalline representation of dimension $d$ over some finite extension $\mathbb{E}$ over $\mathbb{Q}_p$. The $(\varphi, \Gamma)$-module $D_{\text{rig}} (V)$ attached to $V$ is trianguline in $d!$ ways, i.e. there are $d!$ possible filtrations available which make it a trianguline module. Indeed, the set of filtrations on $D_{\text{rig}} (V)$ is in bijection with the refinements of $V$, or equivalently an order on the set of $d!$ eigenvalues of the crystalline Frobenius (see for example prop 2.4.1 in \cite{BC09}). As a consequence, there are $d!$ points (counted with multiplicities because the crystalline eigenvalues could be not all distinct) on $\mathcal{S}^\square_d$ such that when we specialize the trianguline universal family $\mathbb{V}_{\text{tri}}$ to those points we obtain exactly $V$. Although those points share the same rational representation $V$, the specialization of an integral subfamily $\mathbb{T}_{\text{tri}}$ of the universal trianguline family might be different. Of course there are some simple cases, for example if the residual representation of $V$  is irreducible there is a unique Galois stable lattice and there will be actually no difference between the description of the universal integral trianguline family in those points.\\
There are many interesting question around this topic. For example, one could ask if it is true that all the stable lattices in a crystalline representation belong to an integral subfamily of some restriction of the universal trianguline family. The answer to this question is unknown to us. However, we will consider some related examples when the answer is affirmative. Another interesting and challenging question is to understand how the trianguline points attached to $V$ move when $V$ vary $p$-adically. More precisely, we can ask whether is true that if we take another crystalline representation $V'$ which is congruent to $V$ modulo some sufficiently large prime power then some of their trianguline points are close in the trianguline variety. The converse of this is statement holds thanks to the local constancy result that we prove in this article. \\
This question fit in the more general context of understanding, on the $(\varphi, \Gamma)$-modules side, when a pointwise trianguline family of modules admits an analytic triangulation. In this sense, we are asking if we can do that over all the possible choices of a triangulation of a crystalline representation $V$. Bellaiche and Chenevier's theory of refined families answer the question with specific choice of triangulation but they notice that refined families do not behave well under permutations of the parameters (e.g if one changes the triangulation on a generic crystalline representation), as we will explain later. Kedlaya, Pottharst and Xiao strong result (see Cor. 6.3.10 in \cite{KPX14}) answered the question up to a proper birational change of the considered rigid analytic space and moreover concerns locally free families. In our discrete setting, we really want to address the case of free families which essentially interpolate the crystalline points for all possible choices of triangulations. In order to do that, we will require that the points are in some sense close enough, i.e. there are congruences modulo some prime power.\\
A different way of describing this problem is if, instead of dealing with congruences modulo sufficiently high prime powers, we deal with $p$-adic limit of representations. The question then becomes if it is true that if we have a sequence of crystalline representations $V_n$ converging $p$-adically to a crystalline representation $V$ then for some refinement of $V$ we have that there exist refinements of the $V_n$'s whose limit converges in the trianguline variety and is exactly the refined version of $V$. This turns out to be false and in order to correct it we will strengthen the notion of $p$-adic convergence of our sequence. Bellaiche, Khare and Larsen (see \cite{BKL05}), introduced several notions of converging $p$-adic representations. In our case, $p$-adic convergence is what is called uniform physical convergence in \cite{BKL05} and naively speaking means that each entries of every matrix attached to an element of the group $G_{\mathbb{Q}_p}$ converges uniformly as a $p$-adic function. 
It turns out that such notion is too weak to control some integral sequence of lattices inside the sequence of representations (as already observed by Bellaiche, Khare and Larsen) and at the same time it is too weak to keep track of $p$-adic informations on the filtration. This is the reason why we introduce a geometric notion of being congruent modulo prime powers (or equivalently a geometric notion of convergence), which naively speaking means that they are sufficiently close as points in their rigid analytic universal deformation space. We will say that such representations are geometrically congruent modulo some prime power. It turns out that such definition of convergence will allow us to track not only integral structure but $p$-adic information on all refinements as well. In precise terms, we have the following:
\begin{proposition}
Let $V$ and $V'$ two crystalline representations of dimension $d$ defined over a finite extension $\mathbb{E}$ of $\mathbb{Q}_p$. Denote by $\text{Tri}_V$ and $\text{Tri}_{V'}$ the two sets of $d!$ triangulations of respectively $V$ and $V'$ seen as trianguline representations.\\There exists a positive integer $n$ such that if $V$ and  $V'$ are geometrically congruent modulo $p^n$ then:
$$\text{ for all } \tau \in \text{Tri}_V \text{, there exists }\eta_{\tau} \in \text{Tri}_{V'} \text{ such that }D_{V, \tau}, D_{V', \eta_\tau} \in \Omega^{(n)}_{V, \tau} \cap \Omega^{(n)}_{V' , \eta_\tau} ,$$
where $\Omega_{D_{V, \tau}}^{(n)}$ denotes the local constancy modulo $p^n$ neighborhood of the point $D_{V, \tau}$ inside $S_d^{\square,0}$, corresponding to the trianguline $(\varphi, \Gamma)$-module attached to $V$ with rigidified triangulation $\tau$, and the same goes for the point $D_{V', \eta_\tau} $.
\end{proposition} 
\noindent
This result adds some new information concerning the symmetry of the trianguline variety $\mathcal{S}^{\square}_d$. Indeed, in terms of limits of converging sequences of crystalline representation this result can be thought as a discrete approach (in a mathematical sense) in addressing the problem of joining refined crystalline representations via some rigid analytic trianguline family. Bellaiche and Chenevier's theory of refined families of Galois representations (see \cite{BC09}) has, as one if its main goals, to describe when it is possible to create a rigid analytic deformation of a certain fixed nice refined crystalline representation which deforms as well analytically the refinement. This is highly technical and delicate hypothesis on the refinements needs to be considered (for example assuming that the crystalline representation is generic). Here, we only deal with understanding if and how refinements discretely vary in a sequence and how is this related with the existence of analytic integral subfamilies of lattices, however in this weaker discrete framework we do not need to require that the crystalline representations involved are generic. \\
Note also that our result doesn't come directly from applying the theory of refined families in the sense of Bellaiche and Chenevier. Indeed, if one has a refined family of representations joining two refined versions of $V$ and $V'$ it is not possible to permute the parameters of the family in order to have another refined family joining all the other refined versions of $V$ and $V'$ (see Remark 4.2.6 in \cite{BC09}). In other words, refined families do not behave well when we want to permute the analytic parameters. The questions related to the interpretation of our result in terms of converging limit of crystalline representations provided an extra motivation for our work. As a final note, we mention that is not needed to insist on the fact that the limit of the crystalline sequence is itself crystalline if we assume that the Hodge-Tate weights are bounded thanks to a result of Berger (see \cite{Ber04}).\\
\newline
\textbf{Acknowledgments:}  I would like to thank L. Berger for many enlightening conversations and for allowing me to visit him at ENS in Lyon; and I would like to thank as well K. Kedlaya for sharing his insights especially on the last part of this article. 
I would like to thank D. Benois for asking if it was possible to extend the local constancy methods developed for dealing with reductions modulo prime powers in dimension 2 in \cite{Tor22} to a higher dimensional setting. I would like to thank A. Conti for the numerous useful comments on preliminary versions of this article. I would also like to thank G. Wiese and University of Luxembourg for providing support while part of this project was completed. I would like to thank as well A. Rahm and G. Bisson at Universit\'e de la Polyn\'esie fran\c{c}aise which allowed me to carry on the final part of this project. This research was partially supported by the Luxembourg National Research Fund INTER/ANR/18/12589973 GALF and it was partially supported in part by Agence Nationale de la Recherche under grant number ANR-20-CE40-0013.
\section{Integral rigid analytic families of representations of a compact group and their reductions}
Let $p$ be a prime and let $\mathbb{E}$ denote a finite extension of $\mathbb{Q}_p$ whose ring of integers is $\mathcal{O}_\mathbb{E}$.
Let $\mathfrak{X}$ be a quasi-separated, quasi-compact $\mathbb{Q}_p$-rigid analytic space and let $G$ be a compact group. Let $d$ be a positive integer. By a family of representations of $G$ over $\mathfrak{X}$ we mean a locally free (of uniform rank d) coherent $\mathcal{O}_\mathfrak{X}$-module $\mathbb{V}$ (i.e. $\mathcal{O}_\mathfrak{X}$-sheaf) endowed with a continuous $\mathcal{O}_\mathfrak{X}$-linear action of $G$. If we further assume that $\mathfrak{X}$ is an affinoid space then by Kiehl's theorem (see sec. 6.1 in \cite{Bo14}), since the sheaf cohomology over any admissible cover $\mathfrak{U}$ satisfies $H^1 (\mathfrak{U}, \mathbb{V})=0$, we have that there exists a locally free (of uniform rank $d$) $\mathcal{O}_\mathfrak{X} (\mathfrak{X})$-module $V$ endowed with a continuous $\mathcal{O}_\mathfrak{X} (\mathfrak{X})$-linear action of $G$ such that we have a  $G$-invariant isomorphism $\mathbb{V}\cong V \otimes \mathcal{O}_\mathfrak{X}$ of $\mathcal{O}_\mathfrak{X}$-sheaves. In general, if $\mathfrak{X}$ is not affinoid, for a family of representations fo $G$ over $\mathfrak{X}$ we can only say that there exists an admissible affinoid cover $\{\mathfrak{X}_i\}$ such that when the family is restricted to each of the $\mathfrak{X}_i$ then it comes from a $\mathcal{O}_{\mathfrak{X}_i} (\mathfrak{X}_i)$[G]-module; this is precisely the condition for $\mathbb{V}$ to be coherent as a $G$-equivariant $\mathcal{O}_\mathfrak{X}$-module. \\
Since our main objective is to study reductions of representations of $G$, the first step should be to properly define what we mean by integral family of representations of $G$ over $\mathfrak{X}$. In order to do that, we will rely on Raynaud's theory of formal schemes.\\
For any $\mathbb{Q}_p$-rigid analytic space $\mathfrak{X}$ as introduced above, one of the main results of Raynaud's theory, is that there exists an admissible formal $\mathbb{Z}_p$-scheme $\mathcal{X}$ and localized by admissible formal blowing-ups (see for example sec. 8.4 in \cite{Bo14}) whose generic fiber is exactly $\mathfrak{X}$.\\
By an integral family of representations of $G$ over a $\mathbb{Q}_p$-rigid analytic space $\mathfrak{X}$ we mean a locally free (of uniform rank $d$) coherent $\mathcal{O}_{\mathcal{X}}$-module $\mathbb{T}$ with a continuous $\mathcal{O}_{\mathcal{X}}$-linear action of $G$, where $\mathcal{X}$ is a an admissible formal model for $\mathfrak{X}$. Note that, since any admissible formal model $\mathcal{X}$ is in particular locally topologically of finite presentations, the condition for $\mathbb{T}$ of being coherent is equivalent to say that there exists an open affine covering $\{\mathcal{X}_i\}_{i\in J}$ of $\mathcal{X}$ such that the restriction $\mathbb{T}|_{\mathcal{X}_i }$ is associated to a finite $\mathcal{O}_{\mathcal{X}_i} (\mathcal{X}_i)$-module for all $i \in J$ (see sec. 8.1 in \cite{Bo14}). In general, if $\mathcal{X}=\text{Spf}(A)$ is an affine formal scheme and $\mathbb{T}$ is a coherent $\mathcal{O}_{\mathcal{X}}$-module, we cannot directly deduce that $\mathbb{T}$ comes from a $\mathcal{O}_{\mathcal{X}} (\mathcal{X})$-module, or in other words it is not available an analogue of Kiehl's theorem in the context of affine formal geometry in such generality. One has indeed to require that $A$ is a Noetherian adic ring (see Prop. 10.10.5 in \cite{Gr60}) or slightly more generally $A$ is an $R$-algebra topologically of finite presentation where $R$ is a Noetherian adic ring (see Prop. 8.1.5 in \cite{Bo14}). These restrictions will not affect us in the applications because we will consider cases when $A$ is a model of some $\mathbb{Q}_p$-affinoid algebra, i.e. a $\mathbb{Z}_p$-algebra topologically of finite type, or equivalently a quotient of $\mathbb{Z}_p \langle t_1, \dots, t_n\rangle$ by a finitely generated ideal (which is, in particular, a $p$-adic Noetherian ring).\\
Given a family of representations of $G$, say $\mathbb{V}$, over $\mathfrak{X}$ we will show that it is always possible to find an integral subfamily $\mathbb{T}$ with prescribed reductions type. This will be achieved in two different steps. First, we will prove that an integral family always exists (even if, in general, will not come from an integral module). This result is well-known but we will include a proof in the affinoid case for completeness. Second, we will prove that locally it is always possible to find an integral model whose reduction modulo prime powers satisfy a local constancy property. It is interesting to know that research has been carried on in this direction and the most general result is limited to the semi-simple residual case and has been proved in the case of families of representations over adic spaces (see Prop. 5.11 in \cite{He16}). Given an integral family $\mathbb{T}$ of representations of $G$ over an admissible affine formal $\mathbb{Z}_p$-scheme $\mathcal{X}$, it is possible to attach to it a rational family $\mathbb{V}:=\mathbb{T}[\frac{1}{p}]$ of representations of $G$ over $\mathfrak{X}:=\mathcal{X}^{\text{rig}}$ by simply inverting $p$. Indeed, we have a canonical morphism $\mathcal{O}_\mathcal{X} \otimes_{\mathbb{Z}_p} \mathbb{Q}_p  \rightarrow \mathcal{O}_\mathfrak{X}$ (which is an isomorphism under the assumption that the space is normal, this is a result of DeJong, see sec. 7 in \cite{DJ95}) whose image is actually contained in the sheaf $\mathcal{O}^\circ_\mathfrak{X}$ of power bounded functions. Thus, it makes sense to consider integral subfamilies of a given family of representations of $G$ over a $\mathbb{Q}_p$-rigid analytic space $\mathfrak{X}$. Note the subtlety that if one starts with an integral module $\mathbb{T}$ over an admissible formal model, than a reduced hypothesis might be necessary if one still wants $\mathbb{T}$ to be an integral subfamily of $\mathbb{T}[\frac{1}{p}]$ because there might be some torsion obstruction.\\
The following result grants always the existence of an integral model for a relatively large class of rigid analytic spaces (see also Lemma 3.18 in \cite{Che20}):
\begin{proposition}\label{integralfamily}
Let $\mathfrak{X}$ be a quasi-separated, quasi-compact $\mathbb{Q}_p$-rigid affinoid space. Let $\mathbb{V}$ be a family of representations of $G$ over $\mathfrak{X}$.
There exists an admissible $\mathbb{Z}_p$-formal scheme $\mathcal{X}$ (of rigid generic fiber $\mathfrak{X}$) of finite type over $\mathbb{Z}_p$ and $\mathbb{T}$ an integral subfamily of representations of $G$ such that $\mathbb{T}[\frac{1}{p}]\cong \mathbb{V}$ as $\mathcal{O}_\mathfrak{X} [G]$-modules. 
\end{proposition}

\begin{remark} \normalfont
Note that the hypothesis of being quasi-compact and quasi separated are necessary in order to apply Raynaud's theory of formal models. The affinoid hypothesis here is only needed in order to deal with classical modules (instead of general coherent sheaves) which allow us to deliver the idea of the proof easily. On the other hand, contrary to classic algebraic geometry setting, if one wants to deal with classical modules we need the affinoid assumption on $\mathfrak{X}$ in order to have acyclicity of any coherent $\mathcal{O}_{\mathfrak{X}}$-module. Indeed, there are rigid analytic spaces which are quasi-compact and separated but not affinoid (see the very interesting article of Liu \cite{Liu88} for a counterexample).
\end{remark}

\begin{proof}
Let $\mathcal{A}$ be an integral model of $A$. Let $\mathbb{W}$ be a finite and free $\mathcal{A}$-module such that $\mathbb{W}\otimes_{\mathcal{A}} A=\mathbb{V}$. Since $\mathcal{A}$ is open inside $A$, and since $\mathbb{V}$ is a topological finite direct sum of copies of $A$ we have that $\mathbb{W}$ is open inside $\mathbb{V}$.
The action of $G$ can be represented by a continuous map $G \times \mathbb{W} \rightarrow \mathbb{V}$. Since $\mathbb{W}$ is open inside $\mathbb{V}$, then the subgroup $H_\mathbb{W} \subset G$ stabilizing $\mathbb{W}$ is an open subgroup of $G$. Since $G$ is profinite, we have 
that $H_\mathbb{W}$ is of finite index. Let $\{g_i\}_i$ be a finite set of representatives for the left $H_\mathbb{W}$-cosets in $G$. Hence, defining $\mathbb{S}$ as $\sum_i g_i \mathbb{W}$ we have that $\mathbb{S}$ is a $G$-stable, finite $\mathcal{A}$-module such that $\mathbb{S}\otimes_{\mathcal{A}} A=\mathbb{V}$.\\
In general, it is not true that $\mathbb{S}$ is free as a $\mathcal{A}$-module, so in order to find a free module satisfying all the required properties we will work locally using Raynaud's theory of formal models. In particular, we will prove that up to taking a blow-up via a suitable Fitting ideal, it is possible to find a model over which the pull back of the module $\mathbb{S}$ will be free. As a collection of blow-ups can be glued in order to be dominated by a blow-up of the whole formal admissible model (see for example Prop. 14 in Chapt. 8 in \cite{Bo14}), we can proceed locally. \\
Let $V_x$ be an affinoid open neighborhood of $x\in X$. Denote by $A_{V_x}$ its corresponding affinoid algebra and define $\mathbb{V}_{V_x}:=\mathbb{V}\otimes_{A} A_{V_x}$. As every affinoid algebra morphism is in particular a contraction, denote by $\tilde{\mathcal{A}}_{V_x}$ the model for the affinoid algebra $A_{V_x}$ defined by the standard map $A $ to $ A_{V_x}$ corresponding to the inclusion $V_x \hookrightarrow X$. Denote by $\mathfrak{V}_{V_x} = \text{Spf}(\tilde{\mathcal{A}}_{V_x})$ the integral formal scheme (of rigid generic fiber $V_x$) attached to the model $\tilde{\mathcal{A}}_{V_x}$ via Raynaud's theory. We can now restrict our integral subfamilies of $G$-modules to the affinoid open $V_x$ by defining $\mathbb{S}_{V_x}:=\mathbb{S}\otimes_{\mathcal{A}} \tilde{\mathcal{A}}_{V_x} $. We have that $\mathbb{S}_{V_x}$ is a finite $\tilde{\mathcal{A}}_{V_x}$-module with a continuous $\tilde{\mathcal{A}}_{V_x}$-linear action of $G$ and such that the natural scalar extension map $\mathbb{S}_{V_x} \otimes_{\tilde{\mathcal{A}}_{V_x}} A_{V_x} \cong \mathbb{V}_{V_x}$ is an isomorphism of $A_{V_x}[G]$-modules. \\
Now that we have set the problem locally, we can make full use of Raynaud's theory of formal integral models. Indeed, let $\mathcal{I}$ be the Fitting ideal (of $\tilde{\mathcal{A}}_{V_x}$) of the $\tilde{\mathcal{A}}_{V_x}$-module $\mathbb{S}_{V_x}$. The ideal $\mathcal{I}$ defines a blow-up morphism of formal schemes (see sec. 3.3 of \cite{Con08} for the construction): \\
$$\text{Bl}_{\mathcal{I}}: \mathfrak{Z}_{V_x} \rightarrow \mathfrak{V}_{V_x}$$
A priori, the formal scheme $\mathfrak{Z}_{V_x}$ is not affine but up to considering an affine open covering of $\mathfrak{Z}_{V_x}$ and without loss of generality, we can substitute the admissible open $V_x$ with a sufficiently small affine open neighborhood of $x$, say $U_x$, such that the blow-up morphism constructed is actually a morphism between affine formal schemes. Since we are repeating all the above constructions, all the subscripts $V_x$ will now be substituted without loss of generality with $U_x$ affinoid open neighborhood of $x$ contained in $V_x$. This implies that the formal integral scheme $\mathfrak{Z}_{U_x}$ is the formal spectrum of a model $\mathcal{A}_{U_x}$ (a priori, different from $\tilde{\mathcal{A}}_{U_x}$) of the $\mathbb{Q}_p$-affinoid algebra $A$. Indeed, we recall that the blow-up morphism induces an isomorphism on the generic fiber which in this case is the $\mathbb{Q}_p$-rigid analytic affinoid space $U_x=\text{Spm}(A_{U_x})$.
Now, consider the blow-up morphism 
$$\text{Bl}_{\mathcal{I}}: \mathfrak{Z}_{U_x}=\text{Spm}(\mathcal{A}_{U_x}) \rightarrow \mathfrak{V}_{U_x}=\text{Spm}(\tilde{\mathcal{A}}_{U_x}).$$
Identifying the finite $\mathcal{A}_{U_x}$-module $\mathbb{S}$ as a coherent sheaf on the formal scheme $\mathfrak{V}_x$, we can pull it back to a coherent sheaf of the affine formal scheme $\mathfrak{Z}_{U_x}=\text{Spf}(\mathcal{A}_{U_x})$ via the blow-up morphism $\text{Bl}_{\mathcal{I}}: \mathfrak{Z}_{U_x} \rightarrow \mathfrak{V}_{U_x}$, i.e. formally we define $\mathbb{T}_{U_x}:=\text{Bl}^*_{\mathcal{I}} (\mathbb{S}_{U_x})$.  By construction of the blow-up via the Fitting ideal $I$ of the $\tilde{\mathcal{A}}_{U_x}$-module $\mathbb{S}$, the $\mathcal{A}_{U_x}$-module $\mathbb{T}_{U_x}$ will be finite and free and inherits the $\mathcal{A}_{U_x}$-linear action of $G$ compatible with the base change, i.e. we have an isomorphism $\mathbb{T}_{U_x} \otimes_{\mathcal{A}_{U_x}} A_{U_x} \cong \mathbb{V}_{U_x}$ of $A_{U_x}[G]$-modules. This concludes the proof.
\end{proof}

\subsection{Local constancy mod $p^n$ for reductions of integral rigid analytic families}
In the previous section, we have granted the existence of integral subfamilies (essentially seen as integral vector bundles over formal admissible models of the rigid analytic parametrizing space), we will proceed by showing that actually it is possible to prove that such integral model satisfies some local constancy property for any reductions modulo $p^n$. As a first step, we will clarify what we mean by local constancy for reductions of an integral family of representations of $G$ over $\mathfrak{X}$.\\
Let $\mathbb{T}$ be an integral family of representations of $G$ over $\mathcal{X}$ meant in its general form as collection of a coherent sheaf of finite, locally free $\mathcal{O}_\mathcal{X}$-modules over the admissible formal scheme which is the special fiber attached by Raynaud theory to its rigid analytic generic fiber $\mathfrak{X}$ and a set of continuous maps $G\rightarrow \text{Aut}_{\mathcal{O}_\mathcal{X} (\mathcal{U})} (\mathbb{T} (\mathcal{U}))$ for all admissible opens $\mathcal{U}$ inside $\mathcal{X}$. \\
Let $n$ be a positive integer. Since for every $n\geq 1$, the sheaf $p^n \mathcal{O}_\mathcal{X}$ is an open ideal sheaf, it makes sense to consider the quotient sheaf obtained via the sheafification of $\mathcal{O}_\mathcal{X} / p^n \mathcal{O}_\mathcal{X}$. Naturally, we can define $\mathbb{T}^{(n)}$ as the $\mathcal{O}_\mathcal{X}$[G]-module given by the sheafification of the pre-sheaf tensor product $\mathbb{T}\otimes_{\mathcal{O}_\mathcal{X}} (\mathcal{O}_\mathcal{X} / p^n \mathcal{O}_\mathcal{X})$ (whose action of $G$ is given by the action on the left factor). Clearly, we have that $p^n \mathcal{O}_\mathcal{X} \subset \text{Ann}_{\mathcal{O}_\mathcal{X}} (\mathbb{T}^{(n)})$ which is the annihilator of $\mathbb{T}^{(n)}$ as a $\mathcal{O}_\mathcal{X}$-module and the usual commutative algebra properties hold (see for example Lemma 17.23.3. in \cite{SP18}). Note that the tensor product of coherent sheaves is coherent (see for example Lemma 17.16.6 in \cite{SP18}). Note also that one could have equivalently defined $\mathbb{T}^{(n)}$ as the quotient of $\mathcal{O}_\mathcal{X}$-modules given by $\mathbb{T}/p^n \mathbb{T}$. It is straightforward to check that those two sheaves are isomorphic as $\mathcal{O}_\mathcal{X} [G]$-modules. \\
Now, in order to understand better the local constancy phenomena, we need to realize what it means to be congruent modulo $ p^n$ in $\mathcal{O}_\mathcal{X} (\mathcal{U})$ and how congruent elements behave after specialization. In particular, the obvious expectations that specializations of congruent sections modulo $p^n$ are as well congruent modulo $p^n$ (meaning with the same positive integer $n$) is false and in general, one cannot be so optimistic to keep as much informations on the valuations. \\
Before proceeding in an explicit description of congruent sections modulo $p^n$, we need a little basic algebraic detour (see \cite{VW09} for an application in a different context).
\begin{definition}
Let $\mathbb{Q}_p /\mathbb{K} /\mathbb{L}$ be a chain of finite extensions. Let $e_{\mathbb{L}/\mathbb{K}}$ be the ramification index of $\mathbb{L}$ over $\mathbb{K}$. For every positive integer $n$, we define $\gamma_{\mathbb{L}/\mathbb{K}} (n) := (n-1) e_{\mathbb{L}/\mathbb{K}} +1$.
\end{definition}
The function $\gamma_{\mathbb{L}/\mathbb{K}}$ satisfies some useful immediate properties. For $n=1$, we have that $\gamma_{\mathbb{L}/\mathbb{K}}(1)=1$. If $\mathbb{L}$ is an unramified extension of $\mathbb{K}$, then $\gamma_{\mathbb{L}/\mathbb{K}}(n)=n$. If $\mathbb{Q}_p /\mathbb{K} /\mathbb{L}/\mathbb{M}$ is a chain of finite extension, then the multiplicative property (by composition) holds: 
$\gamma_{\mathbb{M}/\mathbb{K}}=\gamma_{\mathbb{M}/\mathbb{L}}\circ \gamma_{\mathbb{L}/\mathbb{K}}$. The integer $\gamma_{\mathbb{L}/\mathbb{K}}$ is the minimal one such that the embedding $\mathcal{O}_\mathbb{K} \hookrightarrow\mathcal{O}_\mathbb{L}$ induces the embedding 
$$\mathcal{O}_\mathbb{K}/\pi^n_\mathbb{K} \mathcal{O}_\mathbb{K} \hookrightarrow\mathcal{O}_\mathbb{L}/\pi^{\gamma_{\mathbb{L}/\mathbb{K}} (n)}_\mathbb{L} \mathcal{O}_\mathbb{L}.$$
As a consequence, for all $\alpha, \beta \in \mathbb{K}$, we have that the usual notion  $\alpha \equiv \beta \mod p^n$ is equivalent to say $ |\alpha-\beta |_p < p^{-n+1} $ which holds if and only if 
$|\alpha-\beta |_\mathbb{K} \leq p^{-\gamma_{\mathbb{K}/ \mathbb{Q}_p } (n)} $.\\
Let $f, g \in \mathcal{O}_\mathcal{U}$, fix a positive integer $n$, fix $x \in \mathcal{U}$ and assume that $\text{ev}_x : \mathcal{O}_\mathcal{U} (\mathcal{U}) \rightarrow \mathcal{O}_\mathbb{E}$ is the specialization map at $x$ (or to be very precise, the specialization map at the point in the chosen special fiber of $\mathcal{U}$ attached to the point $x$) where $\mathcal{O}_\mathbb{E}$ is as usual the ring of integers of the finite extension $\mathbb{E}$ of $\mathbb{Q}_p$ over which the point $x$ is defined.\\
The aim is to understand what it means that the sections $f$ and $g$ are congruent modulo $p^n$ in terms of their specializations. The key is, of course, that the topology of the integral models $\mathcal{O}_\mathcal{U} (\mathcal{U})$ (which are $\mathbb{Z}_p$-algebras topologically of finite type) is the $p$-adic one, i.e. induced by the $p$-adic norm of the attached affinoid algebras and as a consequence, the algebra $\mathcal{O}_\mathcal{U} (\mathcal{U}) /p^n \mathcal{O}_\mathcal{U} (\mathcal{U})$ has the discrete topology. We have that $f$ and $g$  are congruent modulo $p^n$ if for every point $z\in\mathcal{U}$ we have that 
$$|f(z)-g(z)|_p < p^{-n+1} ,$$
now if $z=x$ is defined over the finite extension $\mathbb{E}$ we will have that 
$$|f(x)-g(x)|_{\mathbb{E}/\mathbb{Q}_p} \leq p^{-\gamma_{\mathbb{E}/\mathbb{Q}_p} (n)} ,$$
where $\gamma_{\mathbb{E}/\mathbb{Q}_p} (n)$ is optimal in the sense that is the minimal integer for which the above relation is satisfied.
In other words, we have the following commutative diagram: \\
\begin{center}
\begin{tikzcd}
\mathcal{O}_\mathcal{U} (\mathcal{U})                      \arrow{r}{\text{Pr}_n} \arrow[swap]{d}{\text{ev}_x}                 &             \mathcal{O}_\mathcal{U} (\mathcal{U})/p^n  \mathcal{O}_\mathcal{U} (\mathcal{U})       \arrow{d}{\text{ev}_x} \\
\mathcal{O}_\mathbb{E}  \arrow{r}{Pr_{\gamma_{\mathbb{E}/\mathbb{Q}_p } (n)}} & \mathcal{O}_\mathbb{E} /\pi_{\mathbb{E}}^{\gamma_{\mathbb{E}/\mathbb{Q}_p} (n)}  \mathcal{O}_\mathbb{E} \\
\end{tikzcd}
\end{center}
This allows us to create a direct system of torsion modules over the standard field inclusion and define the ring $\overline{\mathbb{Z}/p^n \mathbb{Z}}:= \underset{\mathbb{K}}{\varinjlim} \mathcal{O}_\mathbb{K} /\pi_{\mathbb{K}}^{\gamma_{\mathbb{K}/\mathbb{Q}_p} (n)}  \mathcal{O}_\mathbb{K}$.\\ For example, for $n=1$, we have the usual algebraic closure $\overline{\mathbb{Z}/p \mathbb{Z}}=\overline{\mathbb{F}}_p$.
By construction, the topological ring  $\overline{\mathbb{Z}/p^n \mathbb{Z}}$ contains all the images of all the specialization maps $\text{ev}_z :  \mathcal{O}_\mathcal{U} (\mathcal{U})/p^n  \mathcal{O}_\mathcal{U} (\mathcal{U}) \rightarrow \mathcal{O}_{\mathbb{K}_z} /\pi_{\mathbb{K}_z}^{\gamma_{\mathbb{K}_z /\mathbb{Q}_p} (n)}  \mathcal{O}_{\mathbb{K}_z}$ for all $z\in\mathcal{U}$, where $\mathbb{K}_z$ is the field of definition of the point $z$.\\
This construction finally leads us into give the following:

\begin{definition}
Let $n$ be a positive integer and let $\mathcal{X}$ be a formal admissible scheme. An integral family of representations $\mathbb{T}$ of $G$ over $\mathcal{X}$ is said to be locally constant modulo $p^n$ if (up to replacing its formal admissible model of definition $\mathcal{X}$ via a blow-up $\mathcal{X}'\rightarrow \mathcal{X}$ and replacing $\mathbb{T}$ with the corresponding pull-back) there exists an affine open covering $\{\mathcal{U}_i \}$ of $\mathcal{X}$ such that the $\mathcal{O}_{\mathcal{U}_i} / p^n \mathcal{O}_{\mathcal{U}_i}$-module $(\mathbb{T}|_{\mathcal{U}_i} )^{(n)}$ it is finite, free of rank $d$ and for all $y,z \in\mathcal{U}_i$, we have an isomorphism as $G$-modules $\text{ev}^*_y (\mathbb{T}^{(n)}) \cong \text{ev}^*_z (\mathbb{T}^{(n)})$ between the push-forwards of the module $\mathbb{T}^{(n)}$ via the specialization maps respectively at $y$ and $z$ composed with the natural inclusion in $\overline{\mathbb{Z} / p^n \mathbb{Z}}$, i.e. $\text{ev}_* : \mathcal{O}_{\mathcal{U}_i} \rightarrow \mathcal{O}_{\mathbb{K}_*} \hookrightarrow \overline{\mathbb{Z} / p^n \mathbb{Z}}$ at any point $*$ with field of definition $\mathbb{K}_*$.
\end{definition}
\begin{remark}\normalfont
We make now a brief technical remark to clarify why the above the definition is at a first look a little bit involved. In all the applications we are 
interested in, one usually has to deal with integral families $\mathbb{T}$ which are lattices inside a rational representation $\mathbb{V}$, 
e.g. $\mathbb{T}$ is a finite and free $\mathcal{O}_\mathfrak{X} (\mathfrak{X})^\circ [G]$-module where $\mathfrak{X}$ is an affinoid rigid 
analytic space. For such families, the notion of local constancy derives directly from the expected norm computation (see for example 
Lemma 2.7). This is the classical definition underlying many computations in the literature (see for example \cite{Ber12} or \cite{Tor22}). 
Equivalently, from a geometric point of view, such notion coincides with saying that there exists an affinoid open cover (not 
a necessarily admissible) of the generic fiber $\mathfrak{X}$ over which the $\mathcal{O}_\mathfrak{X} (\mathfrak{X})^\circ$-submodule $\mathbb{T}$ of $\mathbb{V}$ has locally constant reduction modulo the fixed prime power considered. However, if one wants 
to give a more general definition of being locally constant, say mod $p^n$, for a general integral family $\mathbb{T}$ (seen as a coherent 
module over a generic integral admissible model $\mathcal{X}$) then one cannot proceed in the same way through an affine covering of the model $\mathcal{X}$ directly. 
We clarify this with a short example. Let $\mathcal{X}:= \text{Spf}(\mathbb{Z}_p \langle T \rangle)$ be the formal unit closed ball of 
generic fiber the usual rigid analytic closed affinoid ball $\text{Spm}(\mathbb{Q}_p \langle T \rangle)$. In order to have local constancy, 
one is generally forced to look at points whose norm is bounded by say $p^{-n}$, however the formal model $\text{Spf}(\mathbb{Z}_p \langle T/ p^n 
 \rangle )$ is not an affine open inside $\mathcal{X}$ but rather inside a blow-up $\mathcal{X}'$ of $\mathcal{X}$, namely the blow-up with 
respect to the ideal $(p^n , T)$. The generic case is easily deduced from this one because of Kiel's tubolar neighborhood theorem (see 
the proof in the next section) and from the fact that a collection of local formal blow-ups of a formal admissible model $\mathcal{X}$ (with 
generic fiber $\mathfrak{X}$) can be always glued in order to be dominated by a blow-up of the whole formal admissible model $
\mathcal{X}$ (see for example Prop. 14 in Chapt. 8 in \cite{Bo14}).
This formal correction is the reason why one wants to consider the integral model "up to blow-up" when defining the most general version 
of the notion of locally constant integral families. 
\end{remark}
\begin{remark}\normalfont
Note that introducing the ring $\overline{\mathbb{Z}/p^n \mathbb{Z}}$ was necessary otherwise the previous equality would have made no sense. Note as well the subtle point that one shouldn't be tempted to define the local constancy property modulo $p^n$ by asking that, as a $G$-equivariant sheaf, the $\mathcal{O}_\mathcal{U} /p^n \mathcal{O}_\mathcal{U}$-module $\mathbb{T}^{(n)}$ is the constant one (that would indeed imply that since these are sheaves of $\mathbb{Z}_p$-module, the representation $\mathbb{T}^{(n)}$ would be defined over $\mathbb{Z}/p^n \mathbb{Z}$ which is a much stronger statement). 
\end{remark}
\noindent
One can immediately check that the intuitive expectation for local constancy for reductions mod $p^n$ is satisfied when the points $y$ and $z$ have the same field of definition, say $\mathbb{E}$, which is unramified over $\mathbb{Q}_p$. Indeed, in that case we have an isomorphism:
$$\mathbb{T}_x /p^n \mathbb{T}_x\cong \mathbb{T}_y /p^n \mathbb{T}_y \quad\quad\text{ of   }\;\;\faktor{\mathcal{O}_\mathbb{E}}{p^n \mathcal{O}_\mathbb{E}} \big[G\big]\text{-modules.}$$
\noindent
In the context of adic spaces and adic families of representations, Hellmann proved that the semi-simplification of the residual reduction of an adic family is locally constant. In precise terms, the result is the following (see \cite{He16}): 

\begin{theorem}
Let $\mathfrak{X}$ be an adic space locally of finite type over $\mathbb{Q}_p$ and let $\mathbb{V}$ be a vector bundle on $\mathfrak{X}$ endowed with a continuous action of a topological compact group $G$.\\
At any point $x\in \mathfrak{X}$, the semi-simplification of the residual reduction of the specialization of $\mathbb{V}$ at $x$ is locally constant, i.e. $(\overline{\mathbb{V}\otimes k(x)})^{\text{ss}}$ is locally constant. 
\end{theorem}


\noindent
We are able to extend the above theorem in a much greater generality and the proof is based on a totally different approach which is required because looking at the characteristic polynomials as locally constant functions is not enough as we have explained before.\\
The main result of this section is the following:
\begin{theorem} \label{reductions}
Let $\mathfrak{X}$ be a quasi-separated, quasi-compact $\mathbb{Q}_p$-rigid analytic space. Let $\mathbb{V}$ be a family of representations of $G$ over $\mathfrak{X}$ and let $\mathbb{T}$ be an integral subfamily of representations of $G$ over an integral model $\mathcal{X}$ such that $\mathbb{T}[\frac{1}{p}]\cong \mathbb{V}$ as $\mathcal{O}_\mathfrak{X} [G]$-modules. \\
For every positive integer $n\in\mathbb{Z}_{\geq 1}$, the integral family $\mathbb{T}$ is locally constant modulo $p^n$.
\end{theorem}

\begin{remark}\normalfont  
When $n=1$ we can recover the result of Hellmann (see \cite{He16}) and deduce some extra information as well. Indeed, as we deal directly with integral modules we do not need to apply any semi-simplification process to the residual reduction of the family $\mathbb{T}$. As we will see later in dimension two, this will have many practical interesting consequences. For example, we are able to prove local constancy result when the residual reduction is not semisimple allowing us to distinguish between a Borel reduction type and a completely split reduction type.
\end{remark}


\begin{remark}\normalfont
It is interesting to point out that despite the main goal of this article is to study local constancy phenomena for trianguline representations of $\text{Gal}(\overline{\mathbb{Q}}_p /\mathbb{Q}_p)$, the above result has applications in the global context as well. As this is outside the scope of this article, we will limit ourself to give a generic description of the settings where local constancy modulo $p^n$ arises. As first instance, one could consider the setting of integral models attached to rigid analytic Coleman's families, i.e. analytic families of overconvergent modular forms. In this setting one can construct analytic families of Galois stable lattices attached to a Coleman family and our local constancy results can be applied. See for example the recent preprint \cite{NOR23}.\\
Another example concerns the setting of irreducible automorphic representations attached to unitary groups defined over some totally real number field. Indeed, Chenevier proved a local constancy result for such representations when a certain notion of weight vary $p$-adically (see Theorem 1.3 in \cite{Che20}). In order to be slightly more precise, if $U$ is a sufficiently nice unitary group, Chenevier defines irreducible automorphic representations of $U$ to be congruent modulo $p^n$ if their attached $p$-adic systems of eigenvalues (seen as elements in a certain dual Hecke algebra, in the same spirit of the classical case) are congruent modulo $p^n$ for the classic $p$-adic norm. Chenevier proves (see Theorem 1.6 in \cite{Che20}) that such $p$-adic systems of eigenvalues are parametrized, together with other data, by a rigid analytic space $X$, which admits a a rigid analytic weight map $k : X \rightarrow \mathcal{W}$ (sending essentially each representation $\Pi$ to its weight $k(\Pi)$) where $\mathcal{W}$ is the usual weight space parametrizing multiplicative characters of the maximal compact subgroup $T^\circ$ of the diagonal $\mathbb{Q}_p$-torus of $U$, i.e. $\mathcal{W}:=\text{Hom} (T^\circ ,\mathbb{G}_{\text{m}}^{\text{rig}})$. Chenevier's local constancy result is an immediate consequence of the existence of such analytic morphism, i.e. a direct argument for local constancy because his result does not involve any Galois-stable coherent sheaf but naively speaking $p$-adically close points on a certain moduli space. Indeed, by pulling back a sufficiently small $p$-adic ball neighborhood of $k(\Pi)$ can grant (together with a density result, see Theorem 1.6 in \cite{Che20}) the existence of congruent representations modulo prime powers. 
\end{remark}

\subsection{Proof of Theorem \ref{reductions}}

As the locally constant property is a local property, without loss of generality we can assume that the space $\mathfrak{X}$ is reduced (eventually by replacing it with its underlying reduced space). Moreover, again without loss of generality, we can assume that the integral subfamily $\mathbb{T}$ is a proper $\mathcal{O}^{\circ}_\mathfrak{X} (\mathfrak{X})[G]$-module.\\
Let $\rho$ denote the $G$-representation attached to $\mathbb{V}$ and denote by $\rho^\circ$ the representation attached to its integral subfamily $\mathbb{T}$.
The strategy of the proof consists in explicitly describing a sufficiently small open admissible neighborhood of a fixed point $x\in\mathfrak{X}$ in such a way that the coefficients of the restricted representation $\mathbb{T}$ will be converging integral power series whose valuation after specialization can be controlled via an adapted rigid analytic $p$-adic Lipschitz property. As a consequence we will be able to obtain informations about higher reductions of $\mathbb{T}$ modulo prime powers.
We start by describing locally the space $\mathfrak{X}$ around the fixed point $x\in\mathfrak{X}$. Without loss of generality we can assume that $\mathfrak{X}$ has strictly positive dimension. We will denote by $\mathbb{K}_x$ the field of definition of the point $x\in\mathfrak{X}$ (when the dependence on $x$ is established we might simplify the notation and write directly $\mathbb{K}$). The simplified local description of a neighborhood of $x\in\mathfrak{X}$ is possible thanks to a variant of the Kiehl's tubolar neighborhood theorem as proven by Berger and Chenevier (see Prop. 4.4 and 4.5 in \cite{BC10}), which in precise terms states:

\begin{proposition}
Let $\mathfrak{X}$ be a $\mathbb{Q}_p$-rigid analytic space of dimension strictly positive. If $x\in\mathfrak{X}$ is a regular point, then there exists a rigid analytic morphism $i: \mathbb{B}^{1, +}_{\mathbb{K}_x} \rightarrow \mathfrak{X}$ such that $i$ is a closed immersion and $i(0)=x$ and such that its image is contained in an affinoid neighborhood of $x$. If $x\in\mathfrak{X}$ is not regular, it is still possible to find such morphism $i$ up to replacing the field $\mathbb{K}_x$ with a finite extension.
\end{proposition}
\noindent
For simplicity, we will now assume that $x\in\mathfrak{X}$ is a regular point; the irregular case can easily be settled using the same argument with just the necessity of working eventually with finite extensions. Now that we have an explicit description of a neighborhood of $x$ we can pull back the rational family $\mathbb{V}$ (and its integral subfamily $\mathbb{T}$) via the natural analytic closed immersion $i : \mathbb{B}^{1,+}_{r_x, \mathbb{K}_x} \rightarrow \mathfrak{X}$ for a certain radius $r_x$. Denote the $i$-pull back families families still $\mathbb{V}$ and $\mathbb{T}$. Such representations have coefficients respectively in the reduced Tate algebra $\mathcal{O} (\mathbb{B}^{1,+}_{r_x , \mathbb{K}_x})\cong \mathbb{K}_x \langle \frac{T}{s_x}\rangle $ for some $s_x \in \mathbb{K}^\times$ such that $|s_x |= r_x$, and its integral model which is a subring of power bounded elements $\mathcal{O}^\circ (\mathbb{B}^{1,+}_{r_x , \mathbb{K}_x})\cong \mathcal{O}_{\mathbb{K}_x} \langle \frac{T}{s_x}\rangle$. Now in order to take the integral family $\mathbb{T}$ with coefficients in $\mathcal{O}^\circ (\mathbb{B}^{1,+}_{r_x , \mathbb{K}_x})$ we formally replace $\mathbb{T}$ with its pull back on the blow-up (given by the ideal $(p^n , T)$) which coincide with the affine admissible formal model $\text{Spf}( \mathcal{O}_{\mathbb{K}_x} \langle \frac{T}{s_x}\rangle)$ of the affinoid subdomain $\mathbb{B}^{1,+}_{r_x , \mathbb{K}_x}$ of the unit rigid closed ball $\mathbb{B}^{1, +}_{\mathbb{K}_x}$. After this simplification, it becomes possible to control the valuations of elements in these coefficient rings thanks to the following result, which is essentially a $p$-adic analytic version of the Lipshitz property for analytic functions:

\begin{lemma}
Let $\mathbb{K}$ be a finite extension of $\mathbb{Q}_p$ and let $d\geq 1$ a positive integer. Let $r=(r_i )_{i=1, \dots, d}  \in |\overline{\mathbb{K}}^\times|^d$ and let $z_1 , \dots, z_n \in \overline{\mathbb{K}}^\times$ such that $| z_i | = r_i $. Let $T_r^n = \mathbb{K}\langle \frac{T_1}{z_1}, \dots , \frac{T_d}{z_d}\rangle$ be the $\mathbb{K}$-affinoid Tate algebra attached to the closed $d$-dimensional rigid ball of radius $r$, denoted $\mathbb{B}^{d, +}_r$. For any finite extension $\mathbb{L}$ of $\mathbb{K}$, let $f\in T^n_\rho$ and let $u,v \in\mathbb{B}^{d, +}_r (\mathbb{L})$ then:
$$|f(u)-f(v) | \leq r^\text{max} |f|_\rho  |u-v|, $$
where $r^{\text{max}}=\text{max}_i \; r_i $ and $|f(T)|=| \sum_\nu a_\nu T_1^{\nu_1}\dots T_d^{\nu_d}|= \text{max}_{i, \nu} | a_\nu | r_i^{\nu_i}$.
\end{lemma}
\begin{proof}%
The result follows at once by a direct computation with the definition of the norms involved in the affinoid algebras. See also Prop. 7.2.1 in \cite{BGR84}. 
\end{proof}
\noindent
Thanks to the Kiehl's tubular neighborhood theorem we can simplify further the computations and consider only analytic functions in one variable. The power of this strategy resides in the fact that this computations give rise to an explicit linear local constancy parameter once the residual radius is known. Indeed, in order to exhibit a congruence modulo some prescribed prime power between the representations corresponding to two points, say $u$ and $v$, one just need to verify in formulas that, once a positive integer is fixed $m$, the following diagram commutes: 
$$
\begin{tikzcd}
       & &&\arrow[ld,"{{ev}_{u}}"']\text{Gl}(\mathbb{T} )\cong\text{Gl}_d (\mathcal{O}^\circ (\mathbb{B}^{1,+}_{r_x , \mathbb{K}}))\arrow[rd, "{{ev}_{v}}"] &\\
        G  \arrow["{\rho^\circ}",rrru, bend left=15]  \arrow["{\rho^\circ_v}",rrrr, bend left=38]\arrow["{\rho^\circ_{u}}", rr]&  &\text{Gl}(\mathbb{T}_u )\cong \text{Gl}_d (\mathcal{O}_{\mathbb{K}_u})  \arrow[rd, "{\text{Pr}_m}"] & &  \text{Gl}(\mathbb{T}_v )\cong \text{Gl}_d (\mathcal{O}_{\mathbb{K}_v}) \arrow[ld, "{\text{Pr}_m}"']\\
        & & & \text{Gl}_d (\overline{\mathbb{Z} / p^m \mathbb{Z}}) & \\
        & &  & & 
    \end{tikzcd}
    $$
i.e. for any $g\in G$ we have that $|\rho^\circ_u (g) - \rho^\circ_v (g)| \leq p^m$. Assume now that the radius $r_x $ is of the form $p^r$ for some $r\in \mathbb{Z}$ and assume that $|u-v| \leq p^{-m}$ for some positive integer $m$ with $p^{-m} \leq p^r$.
For any $f \in \mathcal{O}^\circ (\mathbb{B}^{1,+}_{r, \mathbb{K}})$ (i.e. a power bounded element or in other words $|f|_{\text{sup}} \leq 1$), and for the points $u, v \in \mathbb{B}^{1,+}_{p^r , \mathbb{K}}$ we have that the following Lipschitz property is satisfied:
$$|f(x)-f(y)|\leq p^r |x-y|\leq p^{r-m} $$
It is then clear that as long as we require that for a prescribed $n\in\mathbb{Z}_{\geq 1}$, the radius $r_n$ satisfies $r_n :=n+r$, we get for any $u, v$ such that $|u-v| \leq p^{- r_n}$ the congruence modulo $p^{-n}$ we were looking for and this concludes the proof. Note that in all our discussion we never required $u$ and $v$ to be defined over the same field extension as we always work with the $p$-adic norm on a fixed algebraic closure of $\mathbb{Q}_p$.

\section{Rigid analytic families of lattices in trianguline representations and their reductions}

In this section, we will keep the same notations as in Chenevier (see \cite{Che13}) and Hellmann (see also \cite{He16}, even though in this reference families of representations are meant over adic spaces instead of the rigid analytic perspective that we adopt). Now, we will introduce the notion of trianguline, regular and rigidified $(\varphi, \Gamma)$-modules and we will recall their main properties.\\
 Let $d$ be a positive integer. Let $p$ be a prime and consider a $\mathbb{Q}_p$-affinoid algebra $A$. We will denote by $\mathcal{R}_A$ the Robba ring with coefficients in $A$. Let $\textsc{Aff}_{\mathbb{Q}_p}$ be the category of affinoid algebras over $\mathbb{Q}_p$. \\
 Let $\mathcal{T}$ be the $\mathbb{Q}_p$-rigid analytic space parametrizing the continuous characters of $\mathbb{Q}_p^{\times}$. In more precise terms, the rigid analytic space $\mathcal{T}$ represents the functor associating to each $B\in\textsc{Aff}_{\mathbb{Q}_p}$ the set $\mathcal{T}(B):=\text{Hom}_{\text{cont}} (\mathbb{Q}_p^\times , A^{\times})$. It is well known that we have an isomorphism $\mathcal{T}\cong\mathcal{W}\times \mathbb{G}_m$ of $\mathbb{Q}_p$-rigid analytic spaces. Here $\mathcal{W}$ denotes the space characterizing continuous characters of $\mathbb{Z}_p^\times$ and itself consists of a finite disjoint union of open rigid analytic balls. We recall the usual definition of to special characters: the character $x\in \mathcal{T}(\mathbb{Q}_p)$ will denote the identity character and the character $\chi \in\mathcal{T}(\mathbb{Q}_p)$ the cyclotomic character (by local class field theory) satisfying $\chi(p)=1$ and being the identity on $\mathbb{Z}_p^\times$.\\
 Inside the rigid analytic space $\mathcal{T}$ we can identify a special admissible open $\mathcal{T}^{\text{reg}}$ whose points consist in the regular characters in the sense of Colmez and Chenevier (see for example \cite{Che13}). Namely, for any $B\in\textsc{Aff}_{\mathbb{Q}_p}$ a character $\delta :\mathbb{Q}_p^{\times} \rightarrow B^{\times} $ (seen as a point in $\mathcal{T}(B)$) is said to be regular if for every $z\in\text{Sp}(B)$ the character $\delta_z :\mathbb{Q}_p^\times \rightarrow k(z)^\times$ obtained by specialization from $\delta$ to $z$ is not of the form $x^i$ or $\chi x^{-i}$ for any non-negative integer $i$. We want to study trianguline representations of generic dimension $d\geq1$, we will consider the rigid analytic space $\mathcal{T}^d$, i.e. the product of $d$ copies of $\mathcal{T}$. Denote by $\mathcal{T}_d^\text{reg}$ inside $\mathcal{T}^d$ the admissible open defined by the following condition: if $B\in \textsc{Aff}_{\mathbb{Q}_p}$, we define $\mathcal{T}_d^\text{reg} (B) :=\{(\delta_i)_i \in \mathcal{T}^d (B) \;\;: \;\; \delta_i / \delta_j \in \mathcal{T}(A)^\text{reg} \text{ for all }1\leq i <j \leq d\}$. \\
 We have the following (we will keep as much as possible the notation of Chenevier, see \cite{Che13}):
 
 \begin{definition}
 Let $A\in \textsc{Aff}_{\mathbb{Q}_p}$. A regular, trianguline, rigidified $(\varphi, \Gamma)$-module over $\mathcal{R}_A$ is a triple $(D, \text{Fil}(D), \nu)$ where:
 \begin{enumerate}
 \item (regular, trianguline) $(D, \text{Fil}(D))$ is a trianguline $(\varphi, \Gamma)$-module over $\mathcal{R}_A$ where the attached parameters $(\delta_i)$ are inside $\mathcal{T}_d^{\text{reg}} (A)$;
 \item (rigidified) $\nu=(\nu_i)$ is a family of isomorphisms $\nu_i : \text{Fil}_{i+1} (D) / \text{Fil}_{i} (D) \rightarrow \mathcal{R}_A (\delta_i)$ compatible with the $(\varphi, \Gamma)$-actions where $i=0, \dots \text{rank}_{\mathcal{R}_A} (D)-1$.
 \end{enumerate}
 \end{definition}
 \noindent
Define the functor $F_d^\square : \textsc{Aff}_{\mathbb{Q}_p} \rightarrow \textsc{Sets}$ by associating to each $A\in\textsc{Aff}_{\mathbb{Q}_p} $ the set $F_d^\square (A)$ of trianguline, regular and rigidified $(\varphi, \Gamma)$-modules over $A$ up to isomorphism. We have the following key result of Chenevier (see Thm. B in \cite{Che13}):

\begin{theorem}\label{regular}
The functor $F_d^\square$ is representable by a rigid analytic space over $\mathbb{Q}_p$, say $S_d^\square$, which is irreducible, regular and equidimensional of dimension $\frac{d(d+3)}{2}$. 
\end{theorem}
\noindent
It is useful to recall the following result concerning an explicit description of an affinoid open neighborhood of any point in the trianguline rigid space (this is Cor. 3.5 in \cite{Che13}):

\begin{proposition}
For any $x\in S_d^{\square}$ there exists an open affinoid neighborhood $U_x$ inside $S_d^{\square}$ and an open affinoid neighborhood $\Omega_{\delta(x)}$ of $\delta(x)$ inside $\mathcal{T}_d^{\text{reg}}$ and an isomorphism of rigid analytic spaces $i: U_x \rightarrow \Omega \times \mathbb{B}^{\frac{d(d-1)}{2} }$ such that $\text{Pr}_2 \circ i=\delta$.\\
Here $\text{Pr}_2$ denotes the projection map on the second factor and $\mathbb{B}^r$ denotes the rigid analytic closed affinoid ball of dimension $r\in\mathbb{Z}_{\geq1}$.
\end{proposition}
\noindent
It is possible to identify the crystalline points on the trianguline space with some explicit conditions on the parameter characters.
Indeed, let $\mathbb{E}$ be a finite extension of $\mathbb{Q}_p$ and define the set $A_d (\mathbb{E})$ inside $\mathcal{T}_d^{\text{reg}} (\mathbb{E})$ as the set of points $(\delta_i) \in\mathcal{T}^d (\mathbb{E})$ such that:

\begin{enumerate}
\item $\delta_i (p) /\delta_j (p) \not= p^{\pm 1}\text{ for all }i<j;$
\item $\text{ there exist a sequence }(k_i)\in\mathbb{Z} \text{ such that }\delta_i (\gamma)=\gamma^{- k_i}\text{ for all }\gamma\in\Gamma\text{ and }i=1, \dots, d;$
\item $k_1 < k_2< \dots<k_d .$
\end{enumerate}
Then we have the following result of Chenevier (see Lemma 3.15 in \cite{Che13}):
\begin{proposition}
A rigidified, regular, trianguline $(\varphi, \Gamma)$-module over $\mathcal{R}_\mathbb{E}$ such that its attached parameters belong to $A_d (\mathbb{E})$ is crystalline.
\end{proposition}
\noindent
Colmez (in dimension 2) and Chenevier (general dimension) proved that rigidified, regular trianguline $(\varphi, \Gamma)$-modules over some Robba ring arise naturally in analytic families. The way this has been proved consisted of showing the existence of a universal family of such modules which is nothing else then the pull back of the universal family of multiplicative characters of $\mathbb{Q}_p^\times$ via the analytic morphism $\mathcal{S}^{\square}_d \rightarrow \mathcal{T}_d^{\text{reg}}$ induced by the natural morphism of functor $F^\square_d \rightarrow \mathcal{T}_d^\text{reg}$.\\
Our goal is to prove that this universal family, whose local existence is granted by the work of Colmez and Chenevier, admits an integral model whose reductions modulo prime powers can be controlled. 
Chenevier (see Prop. 3.17 in \cite{Che13}) proved that any point in the \'etale locus on the trianguline variety $\mathcal{S}_d^\square$ admits a sufficiently small affinoid open neighborhood, say $\Omega$, over which the unique universal trianguline family $\mathbb{V}$ admits an integral subfamily $\mathbb{T}$ (over some integral model $\mathcal{A}$ of $\mathcal{O} (\Omega)$) which satisfies that $D_{\text{rig}} (\mathbb{T}[\frac{1}{p}] )$ is isomorphic to the universal rigidified trianguline $(\varphi, \Gamma)$-module and the universal property which states that such isomorphism is compatible with any pullback via any analytic affinoid morphism $\psi: Z\rightarrow \Omega$, i.e. the $(\varphi, \Gamma)$-module $D_{\text{rig}} (\psi_* \mathbb{T} [\frac{1}{p}])$ is isomorphic with the rigidified trianguline $(\varphi, \Gamma)$-module induced via pullback from the universal one on $\mathcal{S}_d^\square$ via $Z\rightarrow \mathcal{S}_d^\square$.\\
The main result of this section is the following:

\begin{proposition}
Let $d$ be a positive integer and let $\mathbb{T}:=\mathbb{T}_{\text{tri}}$, defined as a locally free coherent module over $\mathcal{S}^{\square, 0}_d$. For every $x\in S^{\square, 0}_d$ there exists a countable system of open affinoid neighborhoods $\{ \Omega^{(n)} \}_{n\geq 1}$ inside $S^{\square}_d$, a countable compatible set $\{\mathcal{A}^{(n)}\}$ where each $\mathcal{A}^{(n)}$ is a model of the $\mathbb{Q}_p$-affinoid algebra $\mathcal{O}(\Omega^{(n)})$ such that:

\begin{enumerate}
\item (\'etale) $\dots\subset\Omega^{(n)} \subset\Omega^{(n-1)}\subset\dots\subset\Omega=\Omega^{(0)}\subset S^{\square,0}_d$;
\item (specialization) for each $y\in\Omega^{(n)} $, we have an isomorphism $D_{\text{rig}} (\mathbb{T}_y \otimes_{\mathcal{A}^{(n)}} \mathcal{O}(\Omega^{(n)} )_y)\cong D_y $ of regular, rigidified trianguline $(\varphi, \Gamma)$-modules;
\item (local constancy modulo $p^n$) for each $n\geq 1$, we have that $\mathbb{T}$ is locally constant mod $p^n$ on $\Omega$ and constant modulo $p^n$ on $\Omega^{(n)}$.
\end{enumerate}
Here $D_y$ denotes the trianguline $(\varphi, \Gamma)$-module attached to the point $y$, whose definition field is $\mathcal{O}(\Omega^{(n)} )_y$.
\end{proposition}

\begin{remark}\normalfont
Note that the system of models $\{\mathcal{A}_x^{(n)}\}$ is compatible in the sense that there are natural quotient maps $\mathcal{A}_x^{(n)} \rightarrow \mathcal{A}_x^{(n+1)}$ induced by the open immersions $\Omega_x^{(n+1)} \hookrightarrow \Omega_x^{(n)}$. As we will see later, this has the very interesting consequence that there is a choice of a specific lattice for which the local constancy works and such choice does not depend on the fixed positive integer n.
\end{remark}

\begin{remark}\normalfont
The above result has the deep meaning that congruences of local constancy type will always appear in any dimension for a large class of representations which includes the crystalline and the semi-stable ones. In particular, once identified some trianguline parameters, namely any rigid analytic affinoid morphism $\psi: Z \rightarrow \mathcal{S}^{\square}_d$ with image in the \'etale locus, one can pull back the universal family to $Z$ and apply the local constancy results modulo prime powers. Such strategy can be done quite explicitly in some cases. One example of a 1-dimensional rigid analytic family of dimension two representations, can be found in \cite{Tor22} when an explicit rigid analytic affinoid closed immersion $\Phi: Z \rightarrow \mathcal{S}_2$ is defined where $Z$ is the rigid affinoid closed ball centered in 0 and of radius 1 such that $\Phi (1-k)$ correspond to a crystalline point for any $k \in \mathbb{Z}_{\geq 2}$ which allowed the author to prove a local constancy result in the weight modulo prime powers for the crystalline representations of the form $V_{k, a_p}$ with $v_p (a_p) >0$. Another explicit example will be presented in the next section and it concerns a 1-dimensional family of two-dimensional semi-stable representations of fixed Hodge-Tate weights parametrized by the $\mathcal{L}$-invariant where varies in an affinoid subspace inside $\mathbb{P}^{1, \text{rig}}$.

\end{remark}

\begin{proof}
Let $\mathbb{T}$ be the universal integral trianguline family defined by Chenevier (see \cite{Che13}) on the affinoid $\Omega \subset \mathcal{S}_d^{\square, 0}$. The idea is to directly apply the local constancy result proved in the previous section for abstract family for $n=1$. Note the crucial point that thanks to a result of Chenevier (see Theorem \ref{regular}, loc. cit.), the $\mathbb{Q}_p$-rigid analytic space $S_d^{\square, 0}$ is regular and hence when applying the Kiehl's tubolar neighborhood theorem we will not lose any information on the congruence mod $p^n$ due to finite extension. This implies that we can identify an affinoid neighborhood $\Omega^{(1)}$ of $x$ inside $\Omega$ over which the reduction modulo $p$ will be constant. By repeating the same argument, we construct a chain of open affinoids $\{\Omega^{(n)}\}_{n\geq 1}$ which satisfies the property (1). Since $\mathbb{T}$ is the universal family of trianguline lattices, by specializing $\mathbb{T}$ at each point $y\in\Omega^{(n)}$ and after inverting $p$ will give the expected rational trianguline representations corresponding to the point $y$. The property (2) will then be satisfied automatically and the local constancy property will hold by Prop. \ref{reductions} and the fact that every point on $\mathcal{S}_d^{\square, 0}$ is regular.   
\end{proof}
\section{Application of the previous results}
In this section, we are going to apply the local constancy reduction results we obtained previously on the trianguline variety to some specific cases. In the first subsection, we extend to a prime setting and to the whole projective line a local constancy result obtained by Bergdall, Levin and Liu (see \cite{BLL23}) for semi-simple residual reductions of semi-stable representations of $G_{\mathbb{Q}_p}$ of dimension two and of fixed Hodge-Tate weights and large $\mathcal{L}\in \mathbb{P}^1 (\overline{\mathbb{Q}}_p)$. In the second subsection, we study how the triangulations of a crystalline representation $V$ of general dimension $d$ vary when the representation $V$ vary $p$-adically. In particular we prove that if $V$ and $V'$ are sufficiently $p$-adically close then all their refinements will be in some sense $p$-adically close. This addresses the interesting problem of understanding if there exists some $p$-adic integral control over the filtrations of some crystalline trianguline $(\varphi, \Gamma)$-module. 


\subsection{Semi-stable representations of dimension 2}
The goal of this section is to extend to a prime power setting a result of Bergdall, Levin and Liu (see \cite{BLL23}) concerning local constancy for semi-simple residual reduction for semi-stable representation of $G_{\mathbb{Q}_p}$ of dimension 2. First we introduce some notation (as consistent as possible with the one in \cite{BLL23}). Let $x$ and $|x|$ denote respectively the identity character on $\mathbb{Q}_p^\times$ and the character of $\mathbb{Q}_p^\times$ sending $x$ to $p^{- v_p (x)}$. Hence, we have the cyclotomic character $\chi=x |x|$. Let $\pi \in \overline{\mathbb{Q}}_p$ such that $\pi^2 =p$. Let $\mathbb{Q}_{p^2}$ be the unramified quadratic extension of $\mathbb{Q}_p$ of absolute Galois group $G_{\mathbb{Q}_{p^2}}$, denote by $\eta$ its quadratic character modulo $p$ and by $\omega_2$ a fundamental character of level 2 for $\mathbb{Q}_{p^2}$. \\
Let $k\geq 2$ be a positive integer and let $\mathcal{L} \in \overline{\mathbb{Q}}_p$. We define the weakly admissible filtered $(\varphi, N)$-module $D_{k, \mathcal{L}}:= \overline{\mathbb{Q}}_p e_1 \oplus  \overline{\mathbb{Q}}_p e_2$ where:

$$
\varphi=\begin{pmatrix} \pi^k & 0\\ 0 & \pi^{k-2} \end{pmatrix}
\quad
N=\begin{pmatrix} 0 & 0\\ 1 & 0 \end{pmatrix}
\quad 
\text{Fil}_i (D_{k, \mathcal{L}} ) = 
\begin{cases}
D_{k, \mathcal{L}} &\text{ if }i\leq 0 \\
\overline{\mathbb{Q}}_p \cdot(e_1 +\mathcal{L} e_2) &\text{ if }1\leq i \leq k-1\\
0 &\text{ if }k \leq i
\end{cases}
$$
\noindent
By the work of Fontaine and Colmez (see \cite{CF00}), there is an equivalence, say $D_{st}$ between the category of weakly admissible filtered $(\varphi, N)$-modules and the category of semi-stable representations. Hence, let $V_{k, \mathcal{L}}$ be the unique semi-stable non-crystalline representation of Hodge-Tate weights $\{0, k-1\}$ such that $D_{st} (V_{k, \mathcal{L}}^*)=D_{k, \mathcal{L}}$, where the $^*$ denotes as usual the dual representation. Up to twist by a crystalline (or equivalently a semistable) character, all the semi-stable representations are of this form (see \cite{BM02}). The representations $V_{k, \mathcal{L}}$ are always irreducible except when $k=2$.
The parameter $\mathcal{L}$ can be extended to the whole projective line by defining $V_{k, \infty}$ to be the unique crystalline representation of HT weights $\{0, k-1\}$ and of trace of the crystalline Frobenius equal to $\pi^k +\pi^{k-2}$.\\
Bergdall, Levin and Liu (see \cite{BLL23}) proved the following interesting result: 

\begin{proposition}
Assume $k\in\mathbb{Z}_{\geq 4}$, and $p\not=2$. Then if 
$$v_p (\mathcal{L}) < 2-\frac{k}{2}-v_p ((k-2)!),$$
we have the isomorphism $\overline{V}_{k, \mathcal{L}} \cong \overline{V}_{k, \infty} \cong \text{Ind}_{\text{G}_{\mathbb{Q}_{p^2}}}^{\text{G}_{\mathbb{Q}_{p}}} (\omega_2^{k-1} \eta)$.
\end{proposition}
\noindent
The proof relies on an explicit description of the Breuil-Kisin modules (which in some sense generalize Breuil's strongly divisible lattices) attached to the representation $V_{k, \mathcal{L}}$. It is definitely worth mentioning that an approach with strongly divisible modules has been used to get some local constancy result in the context of 3-dimensional representations (see \cite{Par17}).\\
The result above occupies an important place in the literature as it gives in some instances the optimal bound. Other explicit results concerning the description of the semi-simple reductions modulo $p$ (and the universal semi-stable deformation rings attached) have been proven in the case of low Hodge-Tate weights, i.e. less than $p-1$ (see for example the groundbreaking work of Breuil and M\'ezard \cite{BM02} and recent generalizations by Guerberoff and Park \cite{GP19}). The advantage of working with low weights is that one can take full advantage of the theory of strongly divisible modules in order to control the integral representations involved thanks to the work of Liu (see \cite{Liu08}) which solves a conjecture of Breuil (see \cite{Bre02}).\\
In order to not impose such conditions on the weight and to still place ourself in a setting in which is possible to somehow explicitly control the lattices inside semi-stable representations, we need to resort to a different method. The idea indeed is to explicitly describe such semi-stable representations in terms of trianguline representations and interpolate them with the universal family. In order to be precise, we first describe explicitly how we can determine the trianguline data attached to the representation $V_{k, \mathcal{L}}$. According to the work of Colmez (see \cite{Col08}), trianguline representations of dimension two are parametrized by a $\mathbb{Q}_p$-rigid analytic space $\mathcal{S}_2$ whose points are of the form $s=(\delta_1, \delta_2, \mathcal{L})$ where $\delta_1, \delta_2$ are multiplicative characters of $\mathbb{Q}^{\times}_p$ and $\mathcal{L} \in \text{Ext}(\mathcal{R}(\delta_1 ) , \mathcal{R} (\delta_2 )))$. For each point $s$, we denote the trianguline representations attached to it as $V(\delta_1, \delta_2, \mathcal{L})$. We have the following:


\begin{lemma}
Let $V_{k, \mathcal{L}}$ be the semi-stable representation of Hodge-Tate weights $\{0, k-1\}$ and invariant $\mathcal{L} \in\mathbb{P}^{1, \text{rig}} (\mathbb{E})$. We have an isomorphism of trianguline representations:
$$V_{k, \mathcal{L}} \cong V(\delta_{1, k}, \delta_{2, k}, \mathcal{L}),$$
where $\delta_{1, k}=x^{k-1} \alpha^{-1}$, $\delta_{2,k} = |x|^{-1}\alpha^{-1}$ and $\alpha$ is the multiplicative character of $\mathbb{Q}^\times_p$ such that $\alpha(p)=\pi^{-k}$ and sends $\mathbb{Z}_p^\times$ to 1. 
\end{lemma}

\begin{remark}\normalfont
Note that when $k=2$, as predicted, the characters $\delta_{i, k}$ are both \'etale. Indeed,$k=2$ is the only case where the representation $V_{k, \mathcal{\mathcal{L}}}$ is reducible. It is straighforward (using computations in \cite{Col08}) to see that in that case, there is a unique extension of $\delta_{i, k}$ as $G_{\mathbb{Q}_p}$-characters.
\end{remark}
\begin{proof}
We are going to describe the dual of the representation $V_{k, \mathcal{L}}^*$ as a trianguline representation. Notice that it will be straightforward to deduce from that the description of $V_{k, \mathcal{L}}$ as a trianguline representation because $V(\eta_1 , \eta_2 , \mathcal{L})^* = V(\eta_2^{-1} , \eta_1^{-1} , \mathcal{L})$ for all multiplicative characters $\eta_1$ and $\eta_2$ (see for example Prop. 4.4 in \cite{Col08}). Describing $V_{k, \mathcal{L}}$ as a trianguline representation follows as a corollary of a result of Colmez (see Prop. 4.18 in \cite{Col08}). Indeed, the way of proceeding is to associate to compute the Weil-Deligne module $\text{WD}_{k, \mathcal{L}}$ attached to the weakly admissible  filtered $(\varphi, N)$-module $D_{k, \mathcal{L}}$. It is straighforward to check that the $\text{WD}$-module $\text{WD}_{k, \mathcal{L}}$ is determined by the action of the Weil group $W_{\mathbb{Q}_p}$ (i.e. $\rho: W_{\mathbb{Q}_p}\rightarrow \text{Gl} (\text{WD}_{k, \mathcal{L}})$) and by the action of $N$:

$$\rho (g)=\begin{pmatrix}\alpha(g) & 0 \\ 0 & p^{\text{deg}(g)} \alpha(g)\end{pmatrix} \;\;\text{  and  }\;\; N=\begin{pmatrix}0& 0 \\ 1 & 0\end{pmatrix}$$
where $\alpha$ is the multiplicative character of $\mathbb{Q}_p^\times$ satisfying $\alpha(p)=\pi^{-k}$ and sending $\mathbb{Z}_p^\times$ to 1 . Moreover, the above actions satisfy $N(\rho(g))=p^{-\text{deg} (g)}\rho(g) N$ for all $g\in W_{\mathbb{Q}_p}$. Here $\text{deg}(g)\in\mathbb{Z}$ denotes the integer such that $g(x)=x^{p^\text{deg}(x)}$ for all $x\in \overline{\mathbb{F}}_p$. Note the action of $\varphi$ on the filtered $(\varphi, N)$-module can be recovered from $\rho$ by evaluating it on an element of degree $-1$ (say $\text{Frob}^{-1}$), which in particular implies that the eigenvalues of $\varphi$ are the inverse of the ones of $\rho(\text{Frob}^{-1})$.
\end{proof} 
\noindent
Now let $s=(\delta_{1, k} , \delta_{2, k}, \mathcal{L})$ be the trianguline module corresponding to $\mathcal{L} \in \text{Ext} (\mathcal{R}(\delta_{1, k}) , \mathcal{R} (\delta_{2, k}))$ (which by \cite{Col08} is isomorphic to $\mathbb{P}^1 (\mathbb{E})$ where $\mathbb{E}$ is the field of definition of $s$). If $k>2$ the representation $V_{k, \mathcal{L}}$ is irreducible and if it is semi-stable and non-crystalline (i.e. $\mathcal{L}\not=\infty$) a result of Colmez (see Theorem 0.5 in \cite{Col08}) ensure us that $s$ correspond to the unique trianguline $(\varphi, \Gamma)$-module whose attached representation is exactly $V_{k, \mathcal{L}}$. \\
We present the main result of this section: 

\begin{proposition}
Let $k\geq2$ be a fixed positive integer and let $p$ be a fixed prime. Let $\mathcal{L}_1$  be a $\mathbb{E}$-point in $\mathbb{P}^{1, \text{rig}}$, for some finite extension $\mathbb{E}$ of $\mathbb{Q}_p$. Inside $\mathbb{P}^{1, \text{rig}}$, there exists a compatible system of $\mathbb{E}$-subaffinoid closed balls $\{\mathbb{B}_{\mathcal{L}_1 , p^{- r_n}}\}_{n\geq 1}$ centered in $\mathcal{L}_1$ of radius $p^{- r_n} \in |\mathbb{E}^\times |$ and there exists a $G_{\mathbb{Q}_p}$-stable lattice $T_1$ inside $V_{k, \mathcal{L}_1}$ such that 
for any $\mathcal{L}_2 \in \mathbb{B}_{\mathcal{L}_1 , p^{- r_m}}$ then there exist a $G_{\mathbb{Q}_p}$-stable lattice $T_2 \subset V_{k, \mathcal{L}_2}$ which satisfy the following congruence
$$T_1 \equiv T_2 \mod \pi_{\mathcal{L}_2}^{\gamma(m)} \quad\quad \text{ as }\faktor{\mathcal{O}_{\mathbb{K}_{\mathcal{L}_2}}}{\pi_{\mathcal{L}_2}^{\gamma(m)} \mathcal{O}_{\mathbb{K}_{\mathcal{L}_2}}} [G_{\mathbb{Q}_p}]\text{-modules,}$$ 
with $\gamma(m) =e_{\mathbb{K}_{{\mathcal{L}_2}} / \mathbb{E}} (m-1)+1$ where $\mathbb{K}_{\mathbb{L}_2}$ is the finite extension of $\mathbb{E}$ over which the point $\mathcal{L}_2$ is defined.\\
Moreover, the local constancy radiuses satisfy the linear relation $r_n = r_1 +n$ for all $n\geq 1$.
\end{proposition}

\begin{proof}
The result follows directly from pulling back the universal trianguline family via the rigid analytic open immersion $\Phi_k : \mathbb{P}^{1, \text{rig}}\rightarrow \mathcal{S}_2^{\square, 0}$ sending $\Phi_k (\mathcal{L})$ to the unique etale, trianguline $(\varphi, \Gamma)$-module $\mathcal{D}_{k, \mathcal{L}}$ such that $D_{\text{rig}} ( V_{k, \mathcal{L}}^* ) \cong \mathcal{D}_{k, \mathcal{L}} $ and then applying the local constancy results as in the previous section.
\end{proof}
\noindent
Now, as an interesting application in characteristic 2 we can deduce the following corollary. Setting $p=2$ and $\mathcal{L}=\mathcal{\infty}$ we get the first example of local constancy phenomena for reductions in characteristic 2. Note in fact that in most of the works previously mentioned (such as \cite{Ber12}, \cite{Tor22} , \cite{BG15}, \cite{BL20} and \cite{BLL23}), this case is always avoided essentially because the computations become more complicated in characteristic 2 for example the pro-finite group $\Gamma$ is not pro-cyclic anymore. In precise terms we have the following:

\begin{corollary}
Let $n$ and $k$ be positive integers. There exists $l_n \in\mathbb{Z}_{\geq 1}$ such that if $\mathcal{L} \in\mathbb{P}^{1, \text{rig}} (\mathbb{E})$ then 

$$V_{k, \mathcal{L}} \equiv V_{k, \infty} \mod 2^{\frac{\gamma_{\mathbb{E} / \mathbb{Q}_2} (n)}{[\mathbb{E} : \mathbb{Q}_2 ]}} \quad\quad \text{for }v_2 (\mathcal{L}) \geq l_n ,$$
as representations of $\text{Gal}(\overline{\mathbb{Q}}_2 /\mathbb{Q}_2)$.\\
Moreover, all the reductions involved correspond to a fixed lattice in $V_{k, \infty}$.
\end{corollary}

\begin{example}[Explicit construction of locally constant integral families of semi-stable representations whose residual reduction is reducible and non-split]\normalfont
Fix once and for all the parameter $k \geq 5$. For any $\mathcal{L} \in \mathbb{E}$, the representation $V_{k, \mathcal{L}}$ is semi-stable non-crystalline of Hodge-Tate weights $\{0, k-1\}$ and defined over $\mathbb{E}$. Assume that we are in the Fontaine-Laffaille range, i.e. $k-1 < p-1$. The representations of this type have been classified case by case via explicit conditions on the valuations of $k$ and $\mathcal{L}$ thanks to the work of Breuil, Mezard, Guerberoff, Park and Wan (see \cite{BM02}, \cite{GP19} and \cite{LP22}) using the theory of strongly divisible modules. We are going to show how this classification can be used to produce explicit analytic families of stable lattices inside semi-stable representations whose locally constant reduction is reducible and non-split and whose local constancy radius mod $p^n$ is explicit.\\
Keeping a consistent notation with \cite{LP22}, define $a(k-1):= H_{\frac{k-3}{2}} + H_{\frac{k-1}{2}}$ where $H_m$ denotes the $m$-th harmonic number. Assume $\mathcal{L}$ satisfies $v_p (\mathcal{L} - a(k-1))=\frac{4-k}{2}$. One of the results of \cite{LP22}, is that the representation $V_{k, \mathcal{L}}$ admits two (and only two, by a general remark concerning 2-dimensional semi-stable lifts of residual representations with trivial endomorphisms ring) classes of $G_{\mathbb{Q}_p}$-stable non-homothetic lattices. Fix a representative for each class, say $T_{k, \mathcal{L}}$ and $T'_{k, \mathcal{L}}$ and denote respectively by $\rho_{k, \mathcal{L}}$ and $\rho'_{k, \mathcal{L}}$ the attached $G_{\mathbb{Q}_p}$-representations. We have that $\overline{\rho}_{k, \mathcal{L}} := \rho_{k, \mathcal{L}}$ mod $\pi_{\mathbb{E}}$ satisfies:
$$\overline{\rho}_{k, \mathcal{L}} \cong \begin{pmatrix} \mu_{c_{\mathcal{L}}} \overline{\chi}^{k-2} & * \\ 0 & \mu_{d_{\mathcal{L}}} \overline{\chi}\end{pmatrix} \quad\quad \text{ with  }*\not=0,$$
where $c_{\mathcal{L}}=\frac{(k-2) \mathcal{L}}{p^{3-k} \lambda}$ mod $\pi_{\mathbb{E}}$, $d_{\mathcal{L}} = \frac{p}{(k-2) \mathcal{L} \lambda}$ mod $\pi_{\mathbb{E}}$, $\overline{\chi}$ denotes the mod-$p$ cyclotomic character and if $x\in\mathbb{E}^\times$ we have that $\mu_x : G_{\mathbb{Q}_p} \rightarrow \mathbb{E}^\times$ is the unramified character sending $\text{Frob}_p$ to $x$.\\
It is straightforward to compute the reduction of the other integral representation $\rho'_{k, \mathcal{L}}$ by simply applying Ribet's lemma (because $k\geq 5$ and in particular $V_{k, \mathcal{L}}$ is irreducible) and observe that there is a unique non-trivial extension of the characters $\mu_{c_{\mathcal{L}}} \overline{\chi}^{k-2}$ and $\mu_{d_{\mathcal{L}}} \overline{\chi}$ because $\text{dim}_{\mathbb{F}_p} (H^1 (G_{\mathbb{Q}_p}, \overline{\chi}^i))=1$ if $1< i < p-1$. We will focus on the representation $\rho_{k, \mathcal{L}}$ and by the previous observation will be immediate to deduce similar conclusions for $\rho'_{k, \mathcal{L}}$. Note also that $\overline{\rho}^{\text{ss}}_{k, \mathcal{L}}\cong \overline{\rho}'^{\text{ss}}_{k, \mathcal{L}}$\\
Let $\mathcal{Z}_{k}$ be the rigid analytic $\mathbb{E}$-affinoid space defined as all the points $\mathcal{L} \in \mathbb{P}^{1, \text{rig}}_{\mathbb{E}}$ such that $v_p (\mathcal{L}-a(k-1))=\frac{4-k}{2}$. Let $\mathbb{V}_k$ be the restriction of the universal trianguline family to the space $\mathcal{Z}_{k}$. By a result of Colmez (see Theorem 0.5 (iii)in \cite{Col08}), and since the representations $V_{k, \mathcal{L}}$ are all irreducible (because $k > 2$), there exists a unique $\mathbb{E}$-point $\mathcal{L} \in \mathcal{Z}_{k}$ such that specializing $\mathbb{V}_k$ at $\mathcal{L}$ we get exactly $V_{k, \mathcal{L}}$. Because of the explicit description of the coefficients $c_{\mathcal{L}}$ and $d_{\mathcal{L}}$, it is now straightforward to check that if $\mathcal{L}'$ is another element in $\mathcal{Z}_k$ which satisfies $v_p (\mathcal{L}' - \mathcal{L}) \geq 1$ then we have that $\overline{\rho}_{k, \mathcal{L}}\cong \overline{\rho}_{k, \mathcal{L}'}$.\\
Let now $\mathcal{Z}_k^{\overline{\rho}_{k, \mathcal{L}}}$ be the affinoid subspace inside $\mathcal{Z}_k$ consisting of all $\mathcal{L}'$ which satisfy $v_p (\mathcal{L}' - \mathcal{L}) \geq 1$. As a consequence, the universal trianguline family $\mathbb{V}_{\text{tri}}$ admits on the residual locally constant subspace $\mathcal{Z}_k^{\overline{\rho}_{k, \mathcal{L}}}$ an integral family $\mathbb{T}_{\text{tri}}$ which parametrizes the unique class of stable lattices in each $V_{k, \mathcal{L}}$ which give as residual reduction $\overline{\rho}_{k, \mathcal{L}}$. For any fixed $\mathcal{L}$, a direct computation shows that the description of the weakly admissible filtered $(\varphi, N)$-module and the description of the rigidified, trianguline $(\varphi, \Gamma)$-module match (via the identification given by comparing the attached Weil-Deligne representation attached to $V_{k, \mathcal{L}}$). This allows us to connect the trianguline families with the families of lattices coming from families of strongly divisible lattices as defined by Breuil. In particular, thanks to the work of Breuil et al. (see \cite{BM02} and \cite{LP22}) we can identify on $\mathcal{Z}_k^{\overline{\rho}_{k, \mathcal{L}}}$ the family $\mathbb{T}_{\text{tri}}$ with the rigid analytic family of representations obtained by applying the functor $T_{st}$ to the universal analytic family $\mathcal{M}$ of strongly divisible lattices whose residual reduction is the Breuil module corresponding to $\overline{\rho}_{k, \mathcal{L}}$ constructed in \cite{LP22}, i.e. we have an isomorphism $\mathbb{T}_{\text{tri}} \cong T_{st} (\mathcal{M})$ as $\mathcal{O}_{\mathcal{Z}_k^{\overline{\rho}_{k, \mathcal{L}}}} (\mathcal{Z}_k^{\overline{\rho}_{k, \mathcal{L}}}) [G_{\mathbb{Q}_p}]$-modules.\\
Finally, for any positive integer $n$, define $\mathcal{Z}_k^{n , {\overline{\rho}_{k, \mathcal{L}}}}$ as the affinoid subspace of $\mathcal{Z}_k^{\overline{\rho}_{k, \mathcal{L}}}$ of points $\mathcal{L}'$ such that $v_{p} (\mathcal{L}' -\mathcal{L})\geq n$ (with the obvious identification $\mathcal{Z}_k^{1 , {\overline{\rho}_{k, \mathcal{L}}}} = \mathcal{Z}_k^{\overline{\rho}_{k, \mathcal{L}}}$). 
By our local constancy results, it is now straightforward to check that if $\mathcal{L}'$ is a $\mathbb{K}$-point in $\mathcal{Z}_k^{n, {\overline{\rho}_{k, \mathcal{L}}}}$ for some finite extension $\mathbb{K}$ of $\mathbb{E}$, then we have:
$$\faktor{\mathbb{T}_{\text{tri}} (\mathcal{L}')}{p^n \mathbb{T}_{\text{tri}} (\mathcal{L}')} \cong \faktor{\mathbb{T}_{\text{tri}} (\mathcal{L})}{p^n \mathbb{T}_{\text{tri}} (\mathcal{L})} \quad\quad\text { as }\quad\faktor{\mathcal{O}_{\mathcal{Z}_{k}^{n , \overline{\rho}_{k, \mathcal{L}}}}}{p^n \mathcal{O}_{\mathcal{Z}_{k}^{n , \overline{\rho}_{k, \mathcal{L}}}}} [G_{\mathbb{Q}_p}]\text{-modules}$$
or equivalently, 
$$\faktor{T_{k, \mathcal{L}'}}{\pi_{\mathbb{K}}^{\gamma_{\mathbb{K} /\mathbb{E}} (n)} T_{k, \mathcal{L}'}} \cong \faktor{T_{k, \mathcal{L}}}{\pi_{\mathbb{E}}^{n} T_{k, \mathcal{L}}} \quad\quad\text{ as }\quad \faktor{\mathcal{O}_{\mathbb{K}}}{\pi_\mathbb{K}^{\gamma_{\mathbb{K} /\mathbb{E}} (n)} \mathcal{O}_\mathbb{K}} [G_{\mathbb{Q}_p}]\text{-modules},$$
where everything is well defined since we have the embedding $\mathcal{O}_\mathbb{E} / \pi_{\mathbb{E}}^n \mathcal{O}_{\mathbb{E}} \hookrightarrow \mathcal{O}_\mathbb{K} / \pi_{\mathbb{K}}^{\gamma_{\mathbb{K} /\mathbb{E}} (n)} \mathcal{O}_\mathbb{K}$. All the above steps can of course be repeated for the integral representation $\rho'_{k, \mathcal{L}}$ which gives as well of another integral subfamily of the universal trianguline family over the same affinoid, namely $\mathcal{Z}_{k}^{\overline{\rho}_{k, \mathcal{L}}}=\mathcal{Z}_{k}^{\overline{\rho'}_{k, \mathcal{L}}}$, over which the residual reducible non-split reduction will be $\overline{\rho'}_{k, \mathcal{L}}$. This concludes our example. 
\end{example}


\subsection{Crystalline representations of dimension $d$}
Let $\mathbb{E}$ be a finite extension of $\mathbb{Q}_p$. First we recall some basic fact about crystalline filtered $\varphi$-module. Let $V$ be a crystalline representation of dimension $d$ over $\mathbb{E}$ of $G_{\mathbb{Q}_p}:=\text{Gal}(\overline{\mathbb{Q}}_p / \mathbb{Q}_p )$ with distinct Hodge-Tate weights $\{k_1 , k_2 , \dots , k_d \}$. Let $D_V :=D_{\text{rig}} (V)$ its attached crystalline $(\varphi, \Gamma)$-module of rank $d$ over the Robba ring $\mathcal{R}_\mathbb{E}$ (note that what follows can be applied also for non-\'etale crystalline modules but we do not need such general definition, see for example \cite{Che11}). Consider now the attached filtered $\mathbb{E}[\varphi]$-module $D_{\text{cris}} (V)$ and assume that all the eigenvalues of the crystalline Frobenius $\varphi$ acting on $D_{\text{cris}} (V)$ are distinct. Fix a refinement $\mathcal{F}=(\mathcal{F}_i)_{i=1, \dots, d}$ of $V$, i.e. a complete $\mathbb{E}[\varphi]$-stable flag inside $D_{\text{cris}} (V)$ such that:

$$\mathcal{F}_0 =\{0\}\subsetneq \mathcal{F}_1 \subsetneq\dots, \subsetneq \mathcal{F}_d =D_{\text{cris}} (V).$$
For a refinement, being critical essentially means that it is in general position with the Hodge filtration induced on $D_{\text{cris}}$ from $D_{\text{dR}}$. In our case, it is useful to use an equivalent definition for a refinement to be non-critical. Namely,  a refinement is non-critical if and only if the sequence of Hodge-Tate weights attached to it is increasing, hence it is exactly $k_1 < \dots<k_d$ (see prop 2.4.7 in \cite{BC09}).\\
Moreover, the refinement $\mathcal{F}$ satisfies the property that $\varphi$ acts on the rank one quotient $\mathcal{F}_i /\mathcal{F}_{i-1}$ via multiplication of $\varphi_i$ for all $i=1, \dots, d$. Hence, giving a refinement $\mathcal{F}$ is equivalent to give an order to the set of eigenvalues of crystalline Frobenius $\{\varphi_1, \dots, \varphi_d\}$. Every refinement $\mathcal{F}$ induces a triangulation $\mathcal{T}$ on the attached $(\varphi, \Gamma)$-module $D_{\text{rig}} (V)$ over $\mathcal{R}_\mathbb{E}$, namely $\mathcal{T}=(\text{Fil}^i (D_{\text{rig}} (V)))_{i=1, \dots, d}$ where $\text{Fil}^i (D_{\text{rig}} (V))) := \mathcal{R}_\mathbb{E} [1/t]\mathcal{F}_i \cap D_{\text{rig}} (V)$. This essentially implies that the crystalline $(\varphi, \Gamma)$-module $D_V$ is trianguline over $\mathcal{R}_{\mathbb{E}}$ in exactly $d!$ ways. An important result of Bellaiche and Chenevier (see prop 2.4.1 in \cite{BC09}) shows that there is a bijection (whose explicit description is essentially the one above) between the set of refinements $\mathcal{F}$ of $D_{\text{cris}} (V)$ and the set of triangulations of $\mathcal{T}$ inside $D_{\text{rig}} (V)$. We say that $D_V$ is generic (in the sense of Bellaiche and Chenevier, see \cite{BC09} or equivalently defined in \cite{Che13})  if it satisfies the following properties: all the eigenvalues $\{\varphi_i \}$ of the crystalline Frobenius acting on $D_{\text{cris}} (V)$ satisfy $\varphi_i \varphi_j^{-1} \not\in\{1, p^{\pm 1}\}$ for all $i, j$ (in particular they are all distinct) and all of the $d!$ refinements of $D_{\text{cris}} (D_V)$ (corresponding to the $d!$ ordering of the $\varphi_i$'s) are non-critical. Such definition of course extends naturally to all the crystalline points in the etale locus of the trianguline variety $\mathcal{S}_d^{\square, 0}$.\\
Under the assumption that $V$ is generic crystalline, Bellaiche and Chenevier also proved that the parameter $(\delta_i : \mathbb{Q}_p^\times \rightarrow \mathbb{E}^\times)_i$ of the triangulation $\mathcal{T}$ can be expressed in terms of $\mathcal{F}$ via the formulas $\delta_i (p)=\varphi_i p^{-s_i}$ and $\delta_i (\gamma)=\gamma^{-k_i}$ for all $\gamma \in \mathbb{Z}_p^\times$.
So, fixing a crystalline representation $V$ of dimension $d$, is equivalent to fix $d!$ points (possibly with multiplicity) in $\mathcal{S}_d^{\square, 0}$ which can described as the points of the form $s_\mathcal{T} :=(D_V , \mathcal{T})$ where $\mathcal{T}$ varies among the $d!$ triangulations of $D_V = D_{\text{rig}} (V)$.\\
In order to have an example in mind, one can think of the crystalline representation of dimension two of $G_{\mathbb{Q}_p}$ attached to a classical modular cuspidal eigenform of level $N$ coprime with $p$ whose two refinements (or equivalently the two attached points on the trianguline variety $\mathcal{S}_2$) correspond to the two trianguline representations attached to its two $p$-stabilizations in level $pN$.\\
In the general case of a crystalline representation $V$ and its attached $d!$ points on $\mathcal{S}_d^\square$, many interesting questions arise from trying to understand how the universal trianguline family differs in a neighborhood of those points. For example, one could ask if it is true that each stable lattice in $V$ appears as the specialization of an integral analytic subfamily of one of the trianguline families interpolating some refinement of $V$. This is unknown to us in such generality. There are however some known cases. For example, if one assume that the residual reduction is irreducible then it is clear that it has to exists a unique class of stable lattices in $V$ and so the claim is true. Moreover, another special case has been treated in the previous section, when it is shown that when the residual reduction is non-split all the non-homothetic classes of stable lattices in semi-stable representations of dimension two and low Hodge-Tate weights can be interpolated by a trianguline integral family. \\
Another interesting and challenging question concerns how the relative positions of such points vary when the representations $V$ vary $p$-adically. In other words, if $V'$ is another crystalline representation of dimension $d$ such that $V \equiv V' \mod p^n$ for some positive integer $n$, then what can we say about the relative positions of the corresponding two sets of points in the trianguline space $\mathcal{S}_d^{\square}$.\\
As we want the deal with any type of reduction modulo prime powers, it is necessary to establish a basic geometric control over the lattices involved. In order to do that, we first introduce a deformation setting. \\
Let $\mathbb{E}$ be a finite extension of $\mathbb{Q}_p$, let $\mathcal{O}_\mathbb{E}$ its ring of integer and let $k_\mathbb{E}$ be its finite residue field. Let $\overline{r}$ a continuous representation of $G_{\mathbb{Q}_p}$ on a 
$k_\mathbb{E}$-vector space $V_{k_\mathbb{E}}$ of dimension $d$. Fix a basis $\overline{b}$ of $V_{k_\mathbb{E}}$. Consider the functor which associates to each local artinian $\mathcal{O}_\mathbb{E}$-algebra $A$ of residue 
field $k_\mathbb{E}$ the set of equivalence classes of triples $(V_A , \iota, b)$ where $V_A$ is a finite, free $A$-module of rank $d$ endowed with a $A$-linear continuous action of $G_{\mathbb{Q}_p}$, where $\iota: V_A \otimes_{A}
k_\mathbb{E} \rightarrow V_{k_\mathbb{E}}$ is a $G_{\mathbb{Q}_p}$-equivariant isomorphism and $b$ is a basis of $V_A$ such that $\iota(b\otimes 1)=\overline{b}$. This is the functor of the framed deformations of dimension $d$
of the residual representation $\overline{r}$. It is pro-represented by a complete, Noetherian, local $\mathcal{O}_\mathbb{L}$-algebra denoted $R^\square_{\overline{r}}$. There should be no confusion between the square exponent in here and the context of rigidified trianguline modules. We denote by $\mathfrak{X}_{\overline{r}}^\square$ the $\mathbb{E}$-rigid analytic space attached to the formal scheme $\text{Spf}(R^\square_{\overline{r}})$ by the functor
constructed by Raynaud and Berthelot from the category of locally Noetherian formal schemes over $\mathcal{O}_\mathbb{E}$ such that the reduction modulo an ideal of definition is a scheme locally of finite type over $k_\mathbb{E}
$, to the category of rigid analytic spaces over $\mathbb{E}$. For every point $x\in \mathfrak{X}_{\overline{r}}^\square$ denote by $\mathcal{U}^{\text{(n)}}_x$ the local constancy neighborhood modulo $p^n$ defined previously. Note that $\mathcal{U}^{\text{(n)}}_x$ is an open affinoid neighborhood of $x$ defined over $\mathbb{E}$. Note also that the rigid analytic spaces that we are now considering are defined over the finite extension $\mathbb{E}$, so all the results presented in the previous section concerning specializations will hold with $\gamma_{\bullet/\mathbb{E}}$ instead of $\gamma_{\bullet/\mathbb{Q}_p}$. We introduce the following:
\begin{definition}
Let $\mathbb{K}$ be a finite extension of $\mathbb{E}$.Two $\mathbb{K}$-linear representations $V$ and $V'$ are geometrically congruent modulo $p^n$ if there exist $G_{\mathbb{Q}_p}$-stable lattices $T$ and $T'$, respectively inside $V$ and $V'$, such that the corresponding $\mathbb{K}$-points $t$ and $t'$ inside $\mathfrak{X}_{\overline{r}}^\square$ satisfy $t, t' \in \mathcal{V} \subseteq\mathcal{U}^{(n)}_t \cap \mathcal{U}^{(n)}_{t'}$ for some affinoid open subset $\mathcal{V}$.\\
In particular, the isomorphism $\faktor{T}{\pi_\mathbb{K}^{\gamma_{\mathbb{K}/ \mathbb{Q}_p } (n)} T}\cong \faktor{T'}{\pi_\mathbb{K}^{\gamma_{\mathbb{K}/ \mathbb{Q}_p } (n)}  T' }$ holds as $\faktor{\mathcal{O}_\mathbb{K}}{\pi_\mathbb{K}^{\gamma_{\mathbb{K}/ \mathbb{Q}_p } (n)}  \mathcal{O}_\mathbb{K} } [G_{\mathbb{Q}_p}]$-modules.
\end{definition}
\begin{remark}\normalfont
Using the construction that we introduced in the previous section concerning the ring $\overline{\mathbb{Z} /p^n \mathbb{Z}} $, the above definition can easily be extended to representations that are not defined on the same field extension. For simplicity of expression, we will proceed by dealing with representations defined over the same field.
\end{remark}
\noindent
The main result of this section is the following:

\begin{proposition}
Let $V$ and $V'$ two crystalline representations of dimension $d$. Denote by $\text{Tri}_V$ and $\text{Tri}_{V'}$ the two sets of $d!$ triangulations of respectively $V$ and $V'$ seen as trianguline representations.\\There exists a positive integer $n$ such that if $V$ and  $V'$ are geometrically congruent modulo $p^n$ then:

$$\text{ for all } \tau \in \text{Tri}_V \text{, there exists }\eta_{\tau} \in \text{Tri}_{V'} \text{ such that }D_{V, \tau}, D_{ V', \eta_\tau } \in \Omega^{(n)}_{ ( V, \tau )} \cap \Omega^{(n)}_{( V' , \eta_\tau)} ,$$
where $\Omega_{*}^{(n)}$ denotes the local constancy modulo $p^n$ neighborhood of $*$ inside $S_d^{\square,0}$.
\end{proposition} 

\begin{remark}\normalfont
In order to prove this result a first impulse is to try to have some $p$-adic control over the triangulations which is of course related to understanding the relative positions of the $d!$ triangulations of the crystalline representations $V$ and $V'$ in a local admissible formal model of a sufficiently small neighborhood of some refinement of $V$ inside the \'etale locus $S_d^{\square,0}$. However, while in the previous section we recalled an explicit description of an affinoid neighborhood of a crystalline point in $\mathcal{S}_d^\square$, there are no explicit informations on any of its admissible formal models. This is partly due to the lack of integral substructures in the Robba ring. For this reason, such approach seems to us rather difficult to be dealt with explicitly, so we came up with another method in the spirit of deformation theory with respect to a general residual reduction (i.e. not necessarily semi-simple).
\end{remark}

\begin{proof} Denote by $t$ and $t'$ respectively two points corresponding to $V$ and $V'$ on $\mathfrak{X}^\square_{\overline{\rho}}$. We have to prove that there exists an admissible open subset $\mathcal{V}$ of $\mathcal{U}^{(n)}_t$ such that if $t'$ belongs to $\mathcal{V}$ then the statement is true.
The idea is to identify inside the trianguline space an analytic subspace over which the residual representation (a priori not semi-simple) of the universal representation is constant and show, through the existence of a rigid analytic morphism (which will depend on a choice of a family of triangulations) from such subspace to the deformation space, that if we pull back points that are sufficiently close in the deformation space that also in this subspace will be close. For technical reasons, we will not use an admissible open subset in $\mathcal{S}_d^\square$ but some rigid analytic space with a similar structure. Repeating this argument for every choice of triangulation will give locally the wished result. In order to have these ingredients, we will rely on some of the work of Breuil, Hellmann and Schraen (see \cite{BHS17}).\\
Let $\overline{r}: G_{\mathbb{Q}_p} \rightarrow \text{Gl}_d ( k_\mathbb{E} )$ be the residual representation attached to the residual reduction of $V$ and $V'$ obtained by reducing the lattices corresponding to the points in the special fiber above $t$ and $t'$. Let $\mathcal{F}^\square (\overline{r})$ be the functor that attach to each 
reduced rigid analytic space over $\mathbb{E}$, say $\mathfrak{X}$, the equivalence class of $(r, \text{Fil}_\bullet , \delta, \nu)$ where $r : G_{\mathbb{Q}_p} \rightarrow \text{Gl}_d (\mathcal{O}^\circ_\mathfrak{X})$ is a continuous 
morphism such that for every point $z \in \mathfrak{X}$, the reduction of $r\otimes \mathcal{O}_{\mathbb{K}_x}$ (here $\mathbb{K}_x$ is the field of definition of the point $x$, which is in this case a finite extension of $\mathbb{E}$ 
and $ \mathcal{O}_{\mathbb{K}_x}$ is its ring of integers) coincide with $\overline{r}$; the collection $ \text{Fil}_\bullet $ is an increasing filtration of sub-$(\varphi, \Gamma)$-modules 
inside $D_{\text{rig}}^\dag (r) $ which are locally direct summand as $\mathcal{R}_\mathfrak{X}$-modules such that $\text{Fil}_0 =0$ and $\text{Fil}_d = D_{\text{rig}}^\dag (r)$ (here $\mathcal{R}_\mathfrak{X}$ denotes the Robba ring 
over the $\mathbb{E}$-rigid analytic space $\mathfrak{X}$); the parameters $\delta \in \mathcal{T}^d_{\text{reg}} (\mathfrak{X})$ are regular characters over $\mathfrak{X}$ and finally the isomorphisms $\nu$ give the rigidifications $
\nu_i : \text{Fil}_i /\text{Fil}_{i-1} \cong \mathcal{R}_\mathfrak{X} (\delta_i )$ for $i=1, \dots, d$. The functor $\mathcal{F}^\square (\overline{r})$ is representable by a $\mathbb{E}$-rigid analytic space which we denote by $\mathcal{S}
^\square (\overline{r})$ (see sec. 2.2 in \cite{BHS17}). In naive terms, the rigid analytic space $\mathcal{S}^\square (\overline{r})$ parametrizes trianguline integral representation with constant residual reduction. Now we have to clarify what 
is the relation with the subfamilies of lattices inside Chenevier and Colmez's universal trianguline family over the rigid analytic space $\mathcal{S}^\square_d$ which parametrizes regular, rigidified, trianguline modules. First, we point out that there is a little conflict of notation which could mislead 
the reader as it once mislead the author in the tentative of maintaining both the notation used by Chenevier (in \cite{Che13}) and by Breuil, Hellmann and Schraen (in \cite{BHS17}). Indeed, with current 
notation there exists a rigid analytic morphism $\psi: \mathcal{S}^\square (\overline{r}) \rightarrow \mathcal{S}^\square_d$, however this morphism is not an immersion of any kind. In order to define such morphism we proceed by 
steps. First, we identify inside $\mathcal{S}^\square_d$ the open \'etale locus, which in this article and in Chenevier's (see \cite{Che13}) has been denoted $\mathcal{S}^{\square, 0}_d$ (this is called admissible locus $\mathcal{S}^\text{adm}_d$ in the notation of \cite{BHS17}) via the characterizing property of admitting a universal Galois family of representations or in more precise terms, a $G$-stable vector bundle whose specialization at any point gives exactly the correspondent trianguline 
representation in that point. Now, since we want to deal with proper Galois representations over $\text{Gl}_n$, it is necessary to modify the open \'etale locus $\mathcal{S}^{\square,0}_d$ because the universal family $
\mathbb{V}$ is only locally free as a $\mathcal{O}_{{\mathcal{S}^{\square,0}_d}}$-module. This can be done directly by trivializing the universal vector bundle of trianguline representations on $\mathcal{S}^{\square,0}_d$ by viewing it as a $\text{Gl}_n$-torsor and as a consequence obtaining a rigid analytic space $\mathcal{Y}_d$ (which in \cite{BHS17} is denoted by $\mathcal{S}^{\square, \text{adm}}_d$, therefore the conflict of notation with $\mathcal{S}^{\square, 0}_d$ as admissibility here means simply \'etale) together with a rigid analytic morphism $\pi_{\overline{r}} : \mathcal{Y}_d \rightarrow \mathcal{S}^{\square,0}_d$. As a last step, one can verify that the rigid analytic space $\mathcal{S}_d (\overline{r})$ parametrizing the deformations of lattices inside trianguline representations of constant residue reduction $\overline{r}$ identifies with the open subspace of points of $\mathcal{Y}_d$ whose attached representation has reduction $\overline{r}$. To summarize, we defined the following morphism given by the composition:
$$\psi_{\overline{r}} : \mathcal{S}^\square (\overline{r})\hookrightarrow \mathcal{Y}_d \rightarrow \mathcal{S}^{\square,0}_{d,\mathbb{E}} \hookrightarrow \mathcal{S}^{\square}_{d, \mathbb{E}} .$$
Finally, we introduce the space of trianguline deformations $\text{X}^\square_{\text{tri}} (\overline{r})$ of the residual representation $\overline{r}$. Let $\mathcal{T}$ the $\mathbb{Q}_p$-rigid analytic space of multiplicative 
characters of $\mathbb{Q}^\times_p$ and for the finite extension $\mathbb{E}$, let $\mathcal{T}_\mathbb{E} := \mathcal{T} \times_{\mathbb{Q}_p} \mathbb{E}$ be its base change parametrizing its $\mathbb{E}$-points.  Define $
\text{U}^\square_\text{tri} (\overline{r})^\text{reg}$ be the set of points of the form $(x, \delta ) \in \mathfrak{X}^\square_{\overline{r}} \times \mathcal{T}^d_{\text{reg}}$ such that if $r_x$ denotes the $G_{\mathbb{Q}_p}$-representations 
attached to $x$ (compatible with the specialization of the universal deformation family constructed above) then its $(\varphi, \Gamma)$-module $D_\text{rig}^\dag (r_x )$ over the Robba ring $\mathcal{R}_{\mathbb{K}_x }$ is trianguline 
of parameters $\delta=(\delta_1 , \dots, \delta_d )$. We are ready now to define the $\mathbb{E}$-rigid analytic space $\text{X}^\square_{\text{tri}} (\overline{r})$ as the Zariski closure of $\text{U}^\square_\text{tri} 
(\overline{r})^\text{reg}$ inside $\mathfrak{X}^\square_{\overline{r}} \times \mathcal{T}^d_\mathbb{E}$. 
Finally, we have also an analytic morphism 
$\pi_{\overline{r}}:  \mathcal{S}^\square ( \overline{r} ) \rightarrow \mathfrak{X}^\square_{\overline{r}} \times \mathcal{T}^d_{\mathbb{E}}$ sending $(r, \text{Fil}_\bullet , \delta, \nu) \mapsto (r, \delta)$. \\
Moreover, since the image of $\pi_{\overline{r}}$ is exactly $\text{U}^\square_{\text{tri}} (\overline{r})^\text{reg}$ we get 
that $$\pi_{\overline{r}} : \mathcal{S}^\square (\overline{r}) \rightarrow X^\square_{\text{tri}} (\overline{r}) \hookrightarrow \mathfrak{X}^\square_{\overline{r}} \times \mathcal{T}^d_\mathbb{E}. $$
Now, we can compare points coming from deformation theory and the relative positions of the corresponding trianguline lattices. Denote by $\xi_{\overline{r}} : \mathcal{S}^\square (\overline{r}) \rightarrow \mathfrak{X}^\square_{\overline{r}}$ the composition of $\pi_{\overline{r}}$ with the projection on the first factor. We have the following diagram:

 $$
\begin{tikzcd}
 &\mathfrak{X}^\square_{\overline{r}} \times \mathcal{T}^d_\mathbb{E}\arrow[r, "{\text{Pr}_1}"]&\mathfrak{X}^\square_{\overline{r}} \\
  \mathcal{S}^\square (\overline{r}) \arrow[rru, bend right=15,"{\xi_{\overline{r}}}"] \arrow[rd, "{\psi_{\overline{r}}}"] \arrow["{\pi_{\overline{r}}}",ru]& &  \\
 		 		&\mathcal{S}^{\square, 0}_{d, \mathbb{E}} \arrow[r, "{i}",hook]& \mathcal{S}^\square_d\\
 		  		& & 
 \end{tikzcd}
 $$
\noindent
We can proceed now to prove the claim. Let $t$ and $t'$ be the points inside $\mathfrak{X}^\square_{\overline{r}}$ corresponding respectively to the representations $T[\frac{1}{p}]=V$ and $T' [\frac{1}{p}] =V'$. Here $T$ and $T'$ are the lattices whose attached integral representations correspond to the points in the formal model $\text{Spf}(R^\square_{\overline{r}})$ whose image in the generic fiber $\mathfrak{X}^\square_{\overline{\rho}}$ is exactly $t$ and $t'$. The integralrepresentations attached to the lattices $T$ and $T'$ will be the one giving the strong congruence $V \equiv V' $ modulo $p^n$. Let $\sigma_t$ be a fixed triangulation of $D_t := D_\text{rig} (V)$ corresponding to the parameters $\delta_t \in \mathcal{T}^d_\mathbb{E}$. Since $V$ is crystalline, fixing a triangulation $\sigma_t$ is equivalent to fixing an order of the eigenvalues of the crystalline Frobenius acting on $D_{\text{cris}} (V)$. The point $(t, \delta_t)$ belongs to $U^\square_{\text{tri}} (\overline{r})^\text{reg}$ (see also cor. 2.12 in \cite{BHS17}) which is exactly the image of $\pi_{\overline{r}}$. Let $\alpha_t \in \mathcal{S}^\square (\overline{r})$ be the point $(t , \text{Fil}_{\bullet, \sigma_t} , \delta_t , \nu_t)$ (where $\nu_t$ is a fixed rigidification corresponding to the ordered parameters $\delta_t$). The point $\alpha_t$ by construction satisfies $\xi_{\overline{r}} (\alpha_t )= \text{Pr}_1 (\pi_{\overline{r}} (\alpha_t)) = \text{Pr}_1 ((t, \delta_t )) = t$ inside $\mathfrak{X}^\square_{\overline{r}}$. On the other side, we have $\psi_{\overline{r}} (\alpha_t )= (D_t , \text{Fil}_{\bullet, t} , \delta_t , \nu_t )$. We will denote by $s_{\sigma_t}$ the point in $\mathcal{S}^{\square, 0}_{d, \mathbb{E}}$ coinciding with $\psi_{\overline{r}} (\alpha_t )$.\\
As defined before, for any positive integer $n$, let $\mathcal{U}^{(n)}_t$ be the open affinoid neighborhood of $t$ giving the local constancy modulo $p^n$ inside $\mathfrak{X}^\square_{\overline{r}}$ and let $\Omega^{(n)}_{\sigma_t}$ be the open affinoid local constancy mod $p^n$ neighborhood of $s_{\sigma_t}$ inside $\mathcal{S}^{\square,0}_{d,\mathbb{E}}$. Now, define the non-empty open affinoid $W^{(n)}_t := \xi_{\overline{r}}^{-1} (\mathcal{U}^{(n)}_t ) \cap \psi_{\overline{r}}^{-1} (\Omega^{(n)}_{\sigma_t })$ inside $\mathcal{S}^\square (\overline{r})$. If $V$ and $V'$ are sufficiently geometrically congruent modulo $p^n$, we have that there exists a sufficiently small admissible open $Z^{\sigma_t}_{t ,t' }$ (containing both $t$ and $t'$) inside $\mathcal{U}^{(n)}_t$. Up to restriction, we can assume that $\xi^{-1}_{\overline{r}} (Z_{t , t'} )$ is contained in $W^{(n)}_t$. Hence, there exists a triangulation $\tau_{t'}$ of $D_{t'}$ and a corresponding point $\beta_{\tau_{t'}} \in W^{(n)}_t$ such that $\xi_{\overline{r}} (\beta_{\tau_{t'}})=t'$ inside $\mathfrak{X}^\square_{\overline{r}}$. In particular, since the analytic morphism $\xi_{\overline{r}} : \mathcal{S}^\square (\overline{r})\rightarrow \mathfrak{X}^\square_{\overline{r}}$ has to induce locally a morphism on the corresponding formal models, we have that $\beta_{\tau_{t' }} = (t' , \text{Fil}_{\bullet, \tau_{t'}} , \delta_{t'}, \nu_{t'})$ where the parameters $\delta_{t'}$ are uniquely determined by the triangulation $\tau_{t'}$. Indeed, since $V'$ is generic crystalline, the parameters $\delta_{t'}$ depend only on the order of the eigenvalues of the crystalline Frobenius acting on $D_{\text{cris}} (V')$. Finally, since 
$$s_{\tau_{t'}}:=\psi_{\overline{r}} (\beta_{\tau_{t'}}) \in \psi_{\overline{r}} (\xi_{\overline{r}}^{-1} (Z^{\sigma_t}_{t,t'})) \subseteq \psi_{\overline{r}} (W^{(n)}_t ) \subseteq \Omega^{(n)}_{\sigma_t} ,$$
we have found a trianguline point $s_{\tau_{t'}}$ inside $\mathcal{S}^{\square,0}_{d, \mathbb{E}}$ corresponding to the representation $V'$ which belong to the local constancy mod $p^n$ neighborhood of $s_{\sigma_t}$. Repeating the argument for each fixed triangulation $\sigma_t$ among all the $d!$ triangulations available allows us to find the admissible open $Z_{t, t'}:= \cap_{\sigma_t} Z_{t,t'}^{\sigma_t}$ which gives the desired result and the proof is complete.

\end{proof}

\printbibliography[heading=bibintoc,title={References}]

\end{document}